\documentclass[3p]{elsarticle}

\usepackage{lineno}
\usepackage{bm}

\usepackage{subcaption}
\usepackage{diagbox}
\usepackage{pifont}
\usepackage{amsmath}
\newcommand{\cmark}{\ding{51}}%
\usepackage{algorithm,algorithmic}

\newtheorem{theorem}{Theorem}
\newtheorem{proposition}{Proposition}

\newenvironment{proof}[1][Proof]{\begin{trivlist}
\item[\hskip \labelsep {\bfseries #1}]}{\end{trivlist}}

\newenvironment{remark}[1][Remark]{\begin{trivlist}
\item[\hskip \labelsep {\bfseries #1}]}{\end{trivlist}}
\DeclareMathOperator*{\argmin}{arg\,min}

\makeatletter
\newsavebox\myboxA
\newsavebox\myboxB
\newlength\mylenA

\newcommand*\xoverline[2][0.75]{%
    \sbox{\myboxA}{$\m@th#2$}%
    \setbox\myboxB\null
    \ht\myboxB=\ht\myboxA%
    \dp\myboxB=\dp\myboxA%
    \wd\myboxB=#1\wd\myboxA
    \sbox\myboxB{$\m@th\overline{\copy\myboxB}$}
    \setlength\mylenA{\the\wd\myboxA}
    \addtolength\mylenA{-\the\wd\myboxB}%
    \ifdim\wd\myboxB<\wd\myboxA%
       \rlap{\hskip 0.5\mylenA\usebox\myboxB}{\usebox\myboxA}%
    \else
        \hskip -0.5\mylenA\rlap{\usebox\myboxA}{\hskip 0.5\mylenA\usebox\myboxB}%
    \fi}
\makeatother

\usepackage{lipsum}
\usepackage{amsfonts}
\usepackage{graphicx}
\usepackage{epstopdf}
\usepackage{algorithmic}
\ifpdf
  \DeclareGraphicsExtensions{.eps,.pdf,.png,.jpg}
\else
  \DeclareGraphicsExtensions{.eps}
\fi

\let\oldequation\align
\let\oldendequation\endalign

\renewenvironment{align}
  {\linenomathNonumbers\oldequation}
  {\oldendequation\endlinenomath}

\journal{arXiv}

\newcommand{\TheTitle}{Maximum-principle-satisfying second-order Intrusive Polynomial Moment scheme} 

\date{\today}

\usepackage{amsopn}

\def\lambdabar{\xoverline{\bm{\lambda}}}
\def\lambdabarjn{\lambdabar_j^n}

\begin{document}

\begin{frontmatter}

\title{\TheTitle}

\author[adressJonas]{Jonas Kusch}

\author[adressGraham]{Graham W. Alldredge}

\author[adressMartin]{Martin Frank}

\address[adressJonas]{Karlsruhe Institute of Technology, Karlsruhe,
    jonas.kusch@kit.edu}
\address[adressGraham]{FU Berlin, Berlin, graham.alldredge@fu-berlin.de}
\address[adressMartin]{Karlsruhe Institute of Technology, Karlsruhe,
    martin.frank@kit.edu}

\begin{abstract}
Using standard intrusive techniques when solving hyperbolic conservation laws with uncertainties can lead to oscillatory solutions as well as nonhyperbolic moment systems. The Intrusive Polynomial Moment (IPM) method ensures hyperbolicity of the moment system while restricting oscillatory over- and undershoots to specified bounds. In this contribution, we derive a second-order discretization of the IPM moment system which fulfills the maximum principle. This task is carried out by investigating violations of the specified bounds due to the errors from the numerical optimization required by the scheme. This analysis gives weaker conditions on the entropy that is used, allowing the choice of an entropy which enables choosing the exact minimal and maximal value of the initial condition as bounds. Solutions calculated with the derived scheme are nonoscillatory while fulfilling the maximum principle.
The second-order accuracy of our scheme leads to significantly reduced numerical costs.
\end{abstract}


\end{frontmatter}

\section{Introduction}
Hyperbolic conservation laws play an important role in modeling various physical and engineering problems. Examples include the shallow water equations in hydrology as well as the Euler equations in gas dynamics. Finite-volume schemes, which are perhaps the most popular numerical methods for hyperbolic problems, are initially designed for the scalar hyperbolic conservation law,
\begin{align}\label{eq:origProblem-intro}
\partial_t u(t,x)+\partial_x f(u(t,x)) = 0,
\end{align}
because the solution theory of this problem is well established.
The conservation law \eqref{eq:origProblem-intro} is generally supplemented with initial conditions
\begin{align}
 u(t=0,x) = u_0(x)
\end{align}
as well as boundary conditions, though the latter do not play a role in this work.

We wish to determine the solution $u$ which depends on the spatial variable $x\in\mathcal{D}=\mathbb{R}$ and time $t\in\mathbb{R}^+$. The function $f:\mathbb{R}\to\mathbb{R}$ is the system flux.
Since $u$ can become discontinuous even for smooth initial conditions $u_0$, the solution must be seen in the weak sense. To ensure uniqueness, an entropy condition is imposed to pick the physically meaningful weak solution \cite[Chapter~3.8.1]{levequenumerical}.
An important property of such an entropy solution is the maximum principle (see \cite[Chapter~2.4]{holden2015front}), which states that 
\begin{linenomath}\begin{align*}
\min_{x\in\mathcal{D}}u_0(x)\leq u(t,x)\leq \max_{x\in\mathcal{D}}u_0(x)
\end{align*}\end{linenomath}
for all $t$ and $x$. Finite-volume schemes are carefully constructed to satisfy this property on a discrete level, see for example \cite{bell1988unsplit,colella1990multidimensional,liu1993maximum,zhang2011maximum,guermond2014second}.\\

In many practical applications the model parameters and initial conditions are not deterministic, and classical finite-volume methods do not take this into account.
One popular approach for uncertain partial differential equations is the stochastic-Galerkin (SG) method \cite{ghanem2003stochastic}.
It is based on polynomial chaos \cite{wiener1938homogeneous} and promises pseudo-spectral convergence for smooth data \cite{canuto1982approximation}. The key idea is to parameterize the uncertainty with the help of a random variable $\bm\xi\in\bm\Theta\subseteq \mathbb{R}^P$ and span the solution with the help of orthonormal polynomials $\varphi_i$. In the following, we assume a scalar random variable, i.e. $P=1$ and $\xi\in\Theta\subseteq \mathbb{R}$. The solution is then approximated by
\begin{linenomath}\begin{align*}
u(t,x,\xi) \approx \mathcal{U}_{SG}(\bm{u}(t,x))(\xi) = \sum_{i = 0}^N u_i(t,x)\varphi_i(\xi) .
\end{align*}\end{linenomath}
The stochastic-Galerkin ansatz leads to a coupled deterministic system of equations for the expansion coefficients $\bm{u} = (u_0, \dots , u_N)^T$ (which also correspond to moments of the solution).
Simple applications, such as the steady diffusion equation \cite{xiu2002modeling} or the advection equation \cite{gottlieb2008galerkin} show the expected spectral convergence.
However, the solutions to hyperbolic problems are generally nonsmooth, and thus the SG method converges slowly and exhibits the oscillations of Gibbs phenomenon.
In addition to that, the stochastic-Galerkin solution can violate the maximum principle, leading to unphysical solutions. In the case of systems the SG equations may not be hyperbolic, making it impossible to solve with standard methods \cite{despres2013robust}. \\

The Intrusive Polynomial Moment (IPM) method \cite{poette2009uncertainty,poette2011treatment,despres2013robust} is designed to preserve hyperbolicity and is constructed to bound oscillations.
The IPM approach is to replace the stochastic-Galerkin ansatz with one derived from a minimum-entropy principle.
The IPM ansatz has the form
\begin{linenomath}\begin{align*}
u(t,x,\xi) \approx (s')^{-1}\left(\sum_{i = 0}^N \lambda_i(\bm{u}(t,x))\varphi_i(\xi)\right),
\end{align*}\end{linenomath}
where $(s')^{-1}$ is the inverse function of the derivative of a strictly convex entropy density $s$ and $\lambda_i$ are the expansion coefficients which need to be chosen to match moment constraints.
Unlike the SG method, in an IPM method the expansion coefficients of the ansatz do not correspond to the moments of the ansatz, which above we have collected into the vector $\bm{u}$.
Minimum-entropy methods have been used in kinetic theory, where they are sometimes called M$_N$ methods, see for example \cite{levermore1996moment,dubroca2002half,hauck2010positive,hauck2011high}. 
The IPM method has a few key advantages.
First, for a scalar conservation law such as \eqref{eq:origProblem-intro}, the entropy density can be chosen such that the solution only takes values in $(u_-, u_+)$.
The interval $(u_-, u_+)$ can be chosen by the user to, e.g., enforce a maximum principle for the discrete solution.
These bounds on the solution also restrict the under- and overshoots of oscillations.
Other attractive properties possessed by the IPM system are hyperbolicity and entropy dissipation.
However, these nice properties of the IPM method do come at certain costs and challenges. The main cost is computational, since an optimization problem must be numerically solved in every spatial cell at each time step. In order to reduce these costs, one should take advantage of both the parallelizability of the method \cite{garrett2015optimization} and high-order numerical schemes for the resulting moment equations.

One of the main challenges facing the IPM method is that the design of a high-order numerical scheme is more complicated.
Unlike the SG method, the moments $\bm{u}$ of the numerical solution must stay within a certain set, called the realizable set, to ensure that the IPM ansatz can be reconstructed.

Another challenge with IPM methods is that while they successfully dampen oscillations near the bounds $u_-$ and $u_+$, the solutions can still oscillate heavily between these bounds.

We tackle these two challenges in this paper.
After reviewing the derivation of the IPM method in section \ref{sec:Section2}, we give a na\"ive, out-of-the-box numerical method for the IPM equations in section \ref{sec:Section3} to demonstrate the problem of maintaining realizability.
Next, in section \ref{sec:Realizability}, we begin to address the problem of numerically maintaining realizability with a first-order scheme through either a time-step restriction or modification of the numerical method.
In section \ref{sec:highOrder} we extend these results to a second-order scheme.
In section \ref{sec:fd-entropy}, we discuss properties of the minimum-entropy approximation, and introduce a new entropy which leads to smaller oscillations in the solution.
Section \ref{sec:results} presents numerical results for the uncertain Burgers' and advection equations.
Finally in section \ref{sec:Section8}, we summarize our findings and give an outlook on future work.

\section{Stochastic Galerkin and IPM}
\label{sec:Section2}

In this section, we recall the derivations of the stochastic-Galerkin and Intrusive Polynomial Moment systems. Our derivation is carried out for the scalar hyperbolic equation with uncertain initial condition
\begin{subequations}\label{eq:origProblem}
\begin{align}
\partial_t u(t,x,\omega)+\partial_x f(u(t,x,\omega)) &= 0, \\
u(0,x,\omega) &= u_0(x,\omega).
\end{align}
\end{subequations}
Here, $x\in\mathcal{D} = \mathbb{R}$ is the spatial domain and $t\in\mathbb{R}^+$ is time.
The flux $f$ may also depend directly on $\omega$, but for now we suppress this from the notation for clarity.

The stochastic variable is $\omega\in\Omega$, where $\Omega$ is the set of all possible outcomes of a random experiment.
The probability measure $d\mathcal{P}(\omega)$ with $\int_{\Omega}d\mathcal{P}(\omega)=1$ imposes a weighting of different events $\omega$.
Assuming the solution is a second-order random field, we can make use of the generalized polynomial chaos (gPC) approach, which lets us represent the random solution with the help of a random variable $\xi\in\Theta$, which has the probability distribution function $f_{\Xi}$. In the following we abuse notation by writing $u = u(t,x,\xi(\omega))$, and we drop the dependency on $\omega$. From the theory of scalar, hyperbolic problems, we know that the solution satisfies a maximum principle \cite[Chapter~2.4]{holden2015front}
\begin{linenomath}\begin{align*}
\min_{x \in \mathcal{D}, \xi \in \Theta} u_0(x, \xi) \leq u(t, x, \xi)
 \leq \max_{x \in \mathcal{D}, \xi \in \Theta} u_0(x, \xi).
\end{align*}\end{linenomath}
Furthermore, for a fixed $\xi$, any strictly convex function $s :\mathbb{R} \to \mathbb{R}$ is an entropy to this problem, meaning that there exists an entropy flux $h$ satisfying $h'(u) = s'(u)f'(u)$ such that
\begin{linenomath}\begin{align*}
\partial_t s(u) + \partial_x h(u) = 0
\end{align*}\end{linenomath}
for strong solutions. We wish to determine how the uncertainty of the initial condition propagates through the solution over time. In order to derive methods such as stochastic Galerkin, we multiply \eqref{eq:origProblem} with the basis function $\varphi_i$ and the probability distribution function $f_{\Xi}$ and integrate with respect to $\xi$ over $\Theta$. The basis functions are chosen to be orthonormal with respect to $f_{\Xi}$. To simplify the notation, we introduce the bracket operator
\begin{linenomath}\begin{align*}
\langle g \rangle := \int_{\Theta} g(\xi) f_{\Xi}(\xi)d\xi.
\end{align*}\end{linenomath}
The resulting system is then
\begin{subequations}\label{eq:origSys}
 \begin{align}
\partial_t \langle u(t,x,\cdot)\varphi_i\rangle+\partial_x \langle f(u(t,x,\cdot))\varphi_i\rangle &= 0, \\
\langle u(0,x,\cdot) \varphi_i\rangle &= \langle u_0(x,\cdot)\varphi_i\rangle.
\end{align}
\end{subequations}
Since the basis functions are orthonormal, the moments $\langle u \varphi_i\rangle$ can be interpreted as Fourier coefficients.
Provided the solution is sufficiently smooth, these coefficients fall to zero rapidly for increasing order $i$, so the first moments should suffice for a good approximation.
Furthermore, the lower-order moments give the most familiar quantities such as the mean and variance. This motivates using only the first $N+1$ moments to define a discretization of the true solution $u$,
\begin{linenomath}\begin{align*}
u_i(t,x) \approx \langle u(t,x, \cdot ) \varphi_i \rangle\enskip \text{ for } \enskip i = 0,\cdots,N;
\end{align*}\end{linenomath}
we collect these moments and the basis functions into the vectors $\bm{u} = (u_0,\cdots,u_N)^T$ and $\bm{\varphi} = (\varphi_0,\cdots,\varphi_N)^T$, respectively.
The main problem is to find a good ansatz $\mathcal{U}(\bm{u}) \approx u$, which allows us to write \eqref{eq:origSys} as a closed system of equations for the moments $\bm{u}$. For the stochastic-Galerkin method, the ansatz is given by
\begin{linenomath}\begin{align*}
\mathcal{U}_{SG}(\bm{u}(t, x))(\xi) = \sum_{i = 0}^N u_i(t, x)\varphi_i(\xi)
 = \bm{u}(t, x)^T \bm{\varphi}(\xi).
\end{align*}\end{linenomath}
(Note that in kinetic theory, this corresponds to the well-known P$_N$ closure, see for example \cite{chandrasekhar1943stochastic,case1967linear,lewis1984computational,pomraning1973equations}.)
In the vector notation the resulting stochastic-Galerkin system is written as
\begin{subequations}
 \begin{linenomath}\begin{align*}
\partial_t \bm{u}+\partial_x \langle f(\bm{u}^T\bm{\varphi})\bm{\varphi}\rangle &= 0, \\
\bm{u}(0,x) &= \langle u_0(x,\cdot)\bm{\varphi}\rangle.
\end{align*}\end{linenomath}
\end{subequations}
The SG system is attractive because it is relatively cheap to simulate and $\mathcal{U}_{SG}$ converges pseudo-spectrally to the correct solution $u$ for smooth problems.
However, the main drawback is that it exhibits the Gibbs phenomenon, i.e., the solution oscillates heavily for nonsmooth problems.  Also, for systems of hyperbolic equations, the resulting SG system might no longer be hyperbolic, and thus ill-posed. 
An example of a classical SG solution for a hyperbolic problem can be found in Figure~ \ref{fig:SolutionSG}.
\begin{figure}[h!]
\centering
  \includegraphics[width=0.8\linewidth]{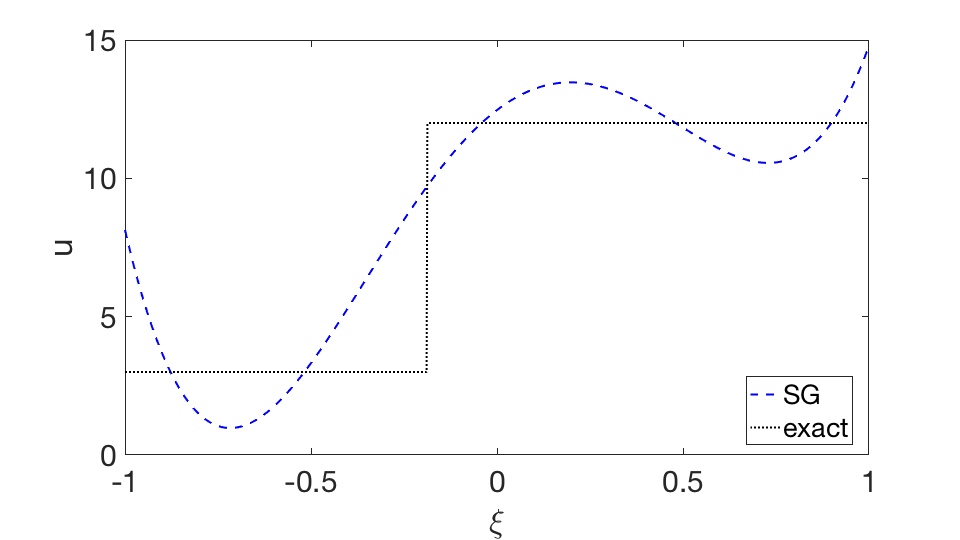}
  \caption{Approximation result for fixed $x$ and $t$ using Burgers' equation (see section \ref{sec:UBurgers}). The SG approximation oscillates heavily and violates the maximum principle.}
  \label{fig:SolutionSG}
\end{figure}

The IPM method, which was introduced in \cite{poette2009uncertainty}, is constructed to overcome these problems.
The idea of the IPM method is to choose the ansatz $\mathcal{U}_{ME}$ that minimizes the convex entropy $\langle s(u) \rangle$ under the moment constraints $\bm{u} = \langle u \bm{\varphi} \rangle$, i.e.,
\begin{align}\label{eq:primalProblem}
\mathcal{U}_{ME}(\bm{u}) = \argmin_{u} \langle s(u) \rangle \enskip \text{ subject to } \bm{u} = \langle u \bm{\varphi} \rangle.
\end{align}
This problem has the unconstrained finite-dimensional dual problem
\begin{align}\label{eq:dualProblem}
 \bm{\hat \lambda}(\bm{u}) := \argmin_{\bm{\lambda} \in \mathbb{R}^{N + 1}}
  \langle s_*(\bm{\lambda}^T \bm{\varphi})\rangle - \bm{\lambda}^T \bm{u},
\end{align}
where $s_*$ is the Legendre transformation of $s$, and $\bm{\lambda}\in\mathbb{R}^{N+1}$ are called the dual variables.
The solution to the primal problem \eqref{eq:primalProblem} is given by
\begin{align}\label{eq:ansatz}
 \mathcal{U}_{ME}(\bm{u}) = \left( s' \right)^{-1}(\bm{\hat{\lambda}}^T \bm{\varphi})
  = s'_*(\bm{\hat{\lambda}}^T \bm{\varphi}).
\end{align}
We also use the notation
\begin{align}
 \mathcal{U}_{ME}(\bm{u}) =: u_{ME}(\hat \Lambda)
\end{align}
for the ansatz, where $u_{ME} = (s')^{-1}$, and $\hat \Lambda = \bm{\hat{\lambda}}^T \bm{\varphi}$ is what we call the dual state.
Inserting this into \eqref{eq:origSys} leads to the closed moment system
\begin{subequations}\label{eq:sysMEU}
\begin{align}
\partial_t \bm{u}+\partial_x \langle f(\mathcal{U}_{ME}(\bm{u}))\bm{\varphi}\rangle &= 0, \\
\bm{u}(0,x) &= \langle u_0(x,\cdot)\bm{\varphi}\rangle.
\end{align} 
\end{subequations}
This is the system of equations of the IPM method.

The IPM system \eqref{eq:sysMEU} has nice features.
First, it generalizes the stochastic-Galerkin method, in that the SG method can be recovered with the quadratic entropy $s(u) =  \frac12 u^2$.
Second, it is hyperbolic for any strictly convex $s$.
Its solutions also satisfy the entropy-dissipation law
\begin{linenomath}\begin{align*}
\frac d{dt} S(t) := \frac d{dt}\int_{\mathcal{D}} \langle s(\mathcal{U}_{ME}(\bm{u}(t,x)))\rangle dx \le 0,
\end{align*}\end{linenomath}
see for example \cite{poette2009uncertainty,kusch2015uncertainty}. Also, with the IPM method one can design the entropy so that the entropy ansatz $\mathcal{U}_{ME}(\bm{u}) = u_{ME}(\hat \Lambda)$ only takes values within a specified interval.
This gives a guaranteed bound on the magnitude of oscillations.
In \cite{poette2009uncertainty} the log-barrier entropy density
\begin{align}\label{eq:log-barrier}
s(u) = -\ln(u-u_-)-\ln(u_+-u)
\end{align}
is used, where the scalars $u_-$ and $u_+$, $u_- < u_+$, are user-specified parameters.
Clearly this entropy does not allow an ansatz which takes on values outside the interval $(u_-, u_+)$.
Since we know that the solution should be bounded by
\begin{linenomath}\begin{align*}
 u_{\min}:=\min_{x,\xi}u_0(x,\xi) 
 \quad \text{and} \quad
 u_{\max}:= \max_{x,\xi}u_0(x,\xi),
\end{align*}\end{linenomath}
one can take $u_+ := u_{\max}+\Delta u$ and $u_- := u_{\min}-\Delta u$ with $\Delta u \in [0, \infty)$.

When the solution to the primal problem \eqref{eq:primalProblem} can only take values in $(u_-, u_+)$, the problem is only feasible if the moment vector $\bm{u}$ lies in the set
\begin{equation}\label{eq:realizableSet}
\mathcal{R} := \left\{ \left. \bm{u}\in\mathbb{R}^{N+1} \right\vert \exists u : \Theta \to (u_-,u_+) \text{ such that } \bm{u} = \langle u \bm{\varphi}\rangle \right\}.
\end{equation}
We call $\mathcal{R}$ the realizable set.
This is important to keep in mind when designing numerical methods, because when the numerical solution leaves the realizable set $\mathcal{R}$, the ansatz $\mathcal{U}_{ME}$ is undefined, and so the IPM method crashes.
We consider this in the next two sections.

\section{Discretization of the IPM system}
\label{sec:Section3}
The IPM system \eqref{eq:sysMEU} can be rewritten as
\begin{linenomath}\begin{align*}
\partial_t \bm{u}+\partial_x \bm{F}(\bm{u}) = \bm{0}
\end{align*}\end{linenomath}
with the flux $\bm{F}(\bm{u})=\langle f(u_{ME}(\hat \Lambda(\bm{u})))\bm{\varphi} \rangle$ depending on the dual state 
\begin{linenomath}\begin{align*}
\hat \Lambda(\bm{u}) = \bm{\hat{\lambda}}(\bm{u})^T\bm{\varphi}.
\end{align*}\end{linenomath}
For efficiency of exposition, we sometimes omit the dependence on $\bm{u}$.
The IPM system is hyperbolic, so it is naturally solved by a finite-volume method. First we discretize the spatial domain into cells. The discrete unknowns are chosen to be the spatial averages over each cell at time $t_n$, given by
\begin{linenomath}\begin{align*}
u_{ij}^n \simeq \frac{1}{\Delta x}\int_{x_{j-1/ 2}}^{x_{j+1/ 2}}u_i(t_n,x) dx.
\end{align*}\end{linenomath}
If a moment vector in cell $j$ at time $t_n$ is denoted as $\bm{u}_j^n = (u_{0j}^n,\cdots,u_{Nj}^n)^T$, the finite-volume scheme can be written in conservative form with the numerical flux $\bm{G}$ as
\begin{align}\label{eq:IPMDiscretization}
\bm{u}_{j}^{n+1} = \bm{u}_{j}^{n}  - \frac{\Delta t}{\Delta x}\left( \bm{G}(\bm{u}_j^n,\bm{u}_{j+1}^n)- \bm{G}(\bm{u}_{j-1}^n,\bm{u}_{j}^n)\right).
\end{align}
The numerical flux is assumed to be consistent, i.e., that $\bm{G}(\bm{u},\bm{u})=\bm{F}(\bm{u})$.
To ensure stability, a CFL condition has to be derived by investigating the eigenvalues of $\nabla \bm{F}$.

When a consistent numerical flux $g = g(u_\ell, u_r)$ is available for the deterministic problem \eqref{eq:origProblem}, then for the IPM system we can simply take
\begin{linenomath}\begin{align*}
 \bm{G}(\bm{u}_{j}^n,\bm{u}_{j+1}^n) = \langle g(u_{ME}(\hat\Lambda_{j}^n),u_{ME}(\hat\Lambda_{j + 1}^n)))\bm{\varphi}\rangle,
\end{align*}\end{linenomath}
where $\hat\Lambda_{j}^n :=\hat\Lambda(\bm{u}_{j}^n)$. This choice of the numerical flux is a common choice in kinetic theory and is called kinetic flux.
The time update of the moment vector now becomes
\begin{align}\label{eq:exactUpdate}
\bm{u}_{j}^{n+1} = \bm{u}_{j}^{n}- \frac{\Delta t}{\Delta x}\left( \langle g(u_{ME}(\hat \Lambda_j^n),u_{ME}(\hat \Lambda_{j+1}^n))\bm{\varphi}\rangle - \langle g(u_{ME}(\hat \Lambda_{j-1}^n),u_{ME}(\hat \Lambda_{j}^n))\bm{\varphi} \rangle\right).
\end{align}

Unfortunately \eqref{eq:exactUpdate} cannot be implemented because the dual problem cannot be solved exactly.%
\footnote{
Equation \eqref{eq:exactUpdate} also includes integral evaluations which cannot be computed in closed form.
Their approximation by numerical quadrature, however, does not play a role in the realizability problems we discuss below.
}
Instead, it must be solved numerically, for example with Newton's method.
The stopping criterion for the numerical optimizer ensures that the approximate multiplier vector it returns, which we denote $\xoverline{\bm{\lambda}}_j^n$ for the moment vector $\bm{u}_j^n$, satisfies the stopping criterion
\begin{align}\label{eq:tauCrit}
\left\Vert \bm{u}_j^n-\left\langle u_{ME}\left(\left(\xoverline{\bm{\lambda}}_j^n\right)^T\bm{\varphi}\right)\bm{\varphi}\right\rangle \right\Vert < \tau.
\end{align}
This is derived from the first-order necessary conditions for the dual problem.
Once the numerical optimizer finds such a $\xoverline{\bm{\lambda}}_j^n$, the corresponding dual state $\xoverline{\Lambda}_j^n := \left(\xoverline{\bm{\lambda}}_j^n\right)^T\bm{\varphi}$ can be used in \eqref{eq:exactUpdate} for the unknown $\hat \Lambda_j^n$.
This gives Algorithm \ref{alg:seq}.

\begin{algorithm}[H]
\begin{algorithmic}[1]
\FOR{$j=0$ to $NCells+1$}
\STATE $\bm{u}_j^0 = \frac{1}{\Delta x} \int_{x_{j-1/ 2}}^{x_{j+1/ 2}} \langle u_0(x, \cdot) \bm{\varphi} \rangle dx$
\ENDFOR
\FOR{$n=0$ to $NTimeSteps$}
\FOR{$j=0$ to $NCells+1$}
\STATE $\xoverline{\bm{\lambda}}_j^n \approx \argmin_{\bm{\lambda}}  \left( \langle s_*(\bm{\lambda}^T \bm{\varphi})\rangle - \bm{\lambda}^T \bm{u}_j^n \right)$
\hfill such that \eqref{eq:tauCrit} holds
\STATE $\xoverline \Lambda_j^n = \left(\xoverline{\bm{\lambda}}_j^n\right)^T\bm{\varphi}$
\ENDFOR
\FOR{$j=1$ to $NCells$}
\STATE $\bm{u}_{j}^{n+1} = \bm{u}_{j}^{n}- \frac{\Delta t}{\Delta x}\left( \langle g(u_{ME}(\xoverline \Lambda_j^n),u_{ME}(\xoverline \Lambda_{j+1}^n))\bm{\varphi}\rangle - \langle g(u_{ME}(\xoverline \Lambda_{j-1}^n),u_{ME}(\xoverline \Lambda_{j}^n))\bm{\varphi} \rangle\right)$ 
\ENDFOR
\ENDFOR
\end{algorithmic}
\caption{IPM for Uncertainty Quantification}
\label{alg:seq}
\end{algorithm}

For most test cases in this paper, Dirichlet boundary conditions are used, i.e. ghost cells  with moment vectors $\bm{u}_{0}^n = \langle u_L \bm{\varphi}\rangle$ and $\bm{u}_{NCells+1}^n = \langle u_R \bm{\varphi}\rangle$ are implemented. Algorithm \ref{alg:seq} crashes when the numerical optimizer cannot find a moment vector $\lambdabarjn$ satisfying the stopping criterion.
This can only%
\footnote{Except for some realizable cases where the problem is so poorly conditioned that the numerical optimizer fails to find the minimizer even though it exists.
See, e.g., \cite{alldredge2012high}.}
happen when the moment vector $\bm{u}_j^n$ is not realizable.
We tested an implementation of Algorithm \ref{alg:seq} on the uncertain Burgers' equation
as described in section \ref{sec:UBurgers}.
We chose the initial condition given in \eqref{eq:IC1}, and ran simulations with different values of the optimization tolerance $\tau$ and the solution-bound parameter $\Delta u$.
We chose the time step $\Delta t$ according to the classical time-step restriction
\begin{equation}
\frac{\Delta t}{\Delta x} \max_{u\in[u_-,u_+]} \vert f'(u)\vert \le 1.
\end{equation}

The solution bounds $u_-$ and $u_+$, which parametrize the entropy \eqref{eq:log-barrier}, are important parameters in the implementation.
Thus one would like to choose $\Delta u$ as small as possible.
Furthermore, since the maximum velocity $\max \{f'(u)\}$ is determined over the interval $u \in [u_-, u_+]$, the larger we take $\Delta u$, the larger the maximum velocity may be.
A larger maximum velocity would lead the CFL condition to impose a tighter time-step restriction and add numerical viscosity.
In \cite{poette2009uncertainty} the authors chose $u_+ = u_{\max}+\Delta u$ and $u_- = u_{\min}-\Delta u$ with $\Delta u = 0.5$. Consequently, over- and undershoots as large as $0.5$ are allowed, and we test a few values here.
We chose all other parameters in the experiments as given in section \ref{sec:UBurgers}.

In Table \ref{tab:convergenceFails}, for different values of the optimization tolerance $\tau$ and the entropy parameter $\Delta u$ we report how long Algorithm \ref{alg:seq} ran until it crashed due to loss of realizability.
The results indicate that this is more likely for smaller values of $\Delta u$, while decreasing the optimization tolerance seems to help slightly.
It is clear, then, that this direct insertion of the numerical optimizer in Algorithm \ref{alg:seq} gives a method which does not preserve realizability.

\begin{table}[tbhp]
\caption{Number of time steps until the dual problem cannot be solved. Check marks indicate successful calculations for all time steps.}
\label{tab:convergenceFails}
\centering
\begin{tabular}{|l|ccccc|}
  \hline
  \diagbox{$\Delta u$}{$\tau$}
               &  $10^{-1}$ & $10^{-2}$ & $10^{-3}$ & $10^{-4}$ & $10^{-5}$ \\
  \hline
  $10^{-1}$ &  \cmark  &    \cmark      &      \cmark      &     \cmark    &   \cmark\\
  $10^{-3}$  & 3  &     3     &      12      &      \cmark   &   \cmark\\
  $10^{-5}$  &  \cmark &      9    &       6     &     8    &   19\\
  \hline
\end{tabular}
\end{table}

In addition to an increased chance of crashes, smaller values of $\Delta u$ can also lead to more oscillatory solutions.
We consider this aspect later in Section \ref{sec:fd-entropy}.
First, we treat the problem of realizability.

\section{Modified scheme to preserve realizability}\label{sec:Realizability}

To understand the reason for the loss of realizability in our tests, we analyze the effects of the optimization error.
It turns out that the optimization error can destroy the monotonicity properties that would otherwise be inherited from the underlying scheme for the original PDE \eqref{eq:origProblem} and would guarantee bounds on the discrete solution.

\subsection{Monotonicity and the optimization error}
\label{sec:Monotonicity}

The main step in Algorithm \ref{alg:seq} is
\begin{align}\label{eq:implementedUpdate}
 \bm{u}_{j}^{n+1} = \bm{u}_{j}^{n}- \frac{\Delta t}{\Delta x}\left( \langle g(u_{ME}(\xoverline \Lambda_j^n),u_{ME}(\xoverline \Lambda_{j+1}^n))\bm{\varphi}\rangle - \langle g(u_{ME}(\xoverline \Lambda_{j-1}^n),u_{ME}(\xoverline \Lambda_{j}^n))\bm{\varphi} \rangle\right).
\end{align}
We can analyze the right-hand side as a function of the dual states $\xoverline{\Lambda}_j^n$ (with the dependence on $\xi$ suppressed) by defining
\begin{align}\label{eq:HfirstOrder}
 H_j^n(\xoverline \Lambda_{j - 1}^n, \xoverline \Lambda_j^n,
   \xoverline \Lambda_{j + 1}^n)
  &:= u_{ME}(\xoverline \Lambda_j^n + \Delta \Lambda_j^n) \\
  &\qquad - \frac{\Delta t}{\Delta x} \left( 
   g(u_{ME}(\xoverline \Lambda_j^n),
   u_{ME}(\xoverline \Lambda_{j+1}^n))
   - g(u_{ME}(\xoverline \Lambda_{j-1}^n),
   u_{ME}(\xoverline \Lambda_{j}^n)) \right), \nonumber
\end{align}
where $\Delta \Lambda_j^n := \hat \Lambda_j^n - \xoverline \Lambda_j^n$.
Thus \eqref{eq:implementedUpdate} can be written as
\begin{align}\label{eq:schemeInexact1}
\bm{u}_{j}^{n+1} =  \left\langle H_j^n(\xoverline{\Lambda}_{j-1}^n,\xoverline{\Lambda}_{j}^n,\xoverline{\Lambda}_{j+1}^n)\bm{\varphi}\right\rangle.
\end{align}
Notice that $H_j^n$ depends itself on the spatial cell and time index because the optimization error, which shows up as $\Delta \Lambda_j^n$, can vary between cells and across time.
We have written $\bm{u}_j^{n + 1}$ simply as the moments of the update function $H_j^n$, so the realizability of $\bm{u}_j^{n + 1}$ can be established by considering whether $H_j^n$ lies in $(u_-, u_+)$.

This leads directly to the concept of monotone schemes for scalar conservation laws, because monotone schemes give numerical solutions which satisfy a maximum principle.
Thus monotonicity can be used to ensure realizability.

\begin{proposition}\label{th:MinMax}
Assume $H_j^n$ is monotonically increasing in each argument.
Then if the entropy ansatz $u_{ME}$ only takes values in $(u_-, u_+)$, the moment vector
$\bm{u}_j^{n + 1}$ computed according to \eqref{eq:schemeInexact1} (or equivalently \eqref{eq:implementedUpdate}) is realizable for any dual states $\xoverline \Lambda_{j - 1}^n$, $\xoverline \Lambda_j^n$, and $\xoverline \Lambda_{j + 1}^n$.
\end{proposition}

\begin{proof}
Let us define $\xoverline{\Lambda}_{j, \max}^n := \max \{ \xoverline{\Lambda}_{j - 1}^n, \xoverline{\Lambda}_{j}^n, \xoverline{\Lambda}_{j + 1}^n \}$. By monotonicity we have for each $\xi$ in each spatial cell
\begin{linenomath}\begin{align*}
H_j^n(\xoverline{\Lambda}_{j - 1}^n, \xoverline{\Lambda}_{j}^n,
  \xoverline{\Lambda}_{j + 1}^n)
 &\leq H_j^n(\xoverline{\Lambda}_{j, \max}^n,
  \xoverline{\Lambda}_{j, \max}^n, \xoverline{\Lambda}_{j, \max}^n)
 = u_{ME}(\xoverline{\Lambda}_{j, \max}^n+\Delta \Lambda_j^n)
 < u_+.
\end{align*}\end{linenomath}
The other direction, $H_j^n > u_-$ can be shown analogously.
Finally, since $H_j^n \in (u_-, u_+)$ for every $\xi$, then $\bm{u}_j^{n + 1} = \langle H_j^n \bm{\varphi} \rangle$ is realizable. \qed
\end{proof}

Now, monotonicity of $H_j^n$ depends on the monotonicity of the scheme defined by the numerical flux $g = g(u_\ell, u_r)$ for the original PDE \eqref{eq:origProblem}.
We assume that under the standard CFL condition,
\begin{equation}\label{eq:CFL}
\frac{\Delta t}{\Delta x} \max_{u\in[u_-,u_+]} \vert f'(u)\vert \le 1,
\end{equation}
$g(u_\ell, u_r)$ gives a monotone scheme for the underlying equation, i.e., that the function
\begin{align}\label{eq:h}
h(u, v, w) = v - \frac{\Delta t}{\Delta x} \left( g(v,w) - g(u,v) \right)
\end{align}
is monotonically increasing in each argument.
This implies
\begin{subequations}
\begin{align}
\frac{\partial g}{\partial u_\ell} &\geq 0, \label{eq:mono1stArg} \\
1-\frac{\Delta t}{\Delta x}\left( \frac{\partial g}{\partial u_\ell}  - \frac{\partial g}{\partial u_r} \right) &\geq 0, \label{eq:mono2ndArg}\\
\frac{\partial g}{\partial u_r} &\leq 0. \label{eq:mono3rdArg}
\end{align}
\end{subequations}
Using this along with properties of the entropy ansatz, we can immediately show that $H_j^n$ is monotone in the first and third arguments, since
\begin{subequations}\label{eq:firstThirdInputH}
\begin{align}
\frac{\partial H_j^n}{\partial \xoverline\Lambda_{j - 1}}
 &= \frac{\Delta t}{\Delta x} \frac{\partial g}{\partial u_\ell}
  u_{ME}'(\xoverline \Lambda_{j-1}^n)
 = \frac{\Delta t}{\Delta x} \underbrace{
  \frac{\partial g}{\partial u_\ell}}
  _{\substack{\geq 0 \\ \text{ by \eqref{eq:mono1stArg}}}}
  \underbrace{\frac{1}{s''(\xoverline \Lambda_{j - 1}^n)}}
  _{\substack{\geq 0 \\ \text{ by convexity}}} \geq 0, \\
\frac{\partial H_j^n}{\partial \xoverline\Lambda_{j + 1}}
 &= -\frac{\Delta t}{\Delta x}
  \underbrace{\frac{\partial g}{\partial u_r}}
  _{\substack{\leq 0 \\ \text{ by \eqref{eq:mono3rdArg}}}}
  \underbrace{\frac{1}{s''(\xoverline \Lambda_{j+1}^n)}}
  _{\substack{\geq 0 \\ \text{ by convexity}}}\geq 0.
\end{align}
\end{subequations}
These properties hold for \emph{any} value of $\Delta \Lambda _j^n$.
But in the second argument, the optimization error $\Delta \Lambda _j^n$ can destroy monotonicity:
\begin{subequations}\label{eq:dHdLambdaj}
\begin{align}\label{eq:monI}
\frac{\partial H_j^n}{\partial \xoverline\Lambda_{j}}
 &= \frac1{s''(\xoverline{\Lambda}_{j}^n + \Delta\Lambda_j^n)}
  - \frac{\Delta t}{\Delta x} \left( \frac{\partial g}{\partial u_\ell}
  \frac{1}{s''(\xoverline{\Lambda}_{j}^n)}
  - \frac{\partial g}{\partial u_r}
  \frac{1}{s''(\xoverline{\Lambda}_{j}^n)} \right) \\
 &= \frac1{s''(\hat{\Lambda}_{j}^n)}
  \left(1 - \frac{s''(\hat{\Lambda}_{j}^n)}{s''(\xoverline{\Lambda}_{j}^n)}
  \frac{\Delta t}{\Delta x} \left(
  \frac{\partial g}{\partial u_\ell} - \frac{\partial g}{\partial u_r}
  \right) \right). \label{eq:noMono}
\end{align}
\end{subequations}
The $s''$ factor in front is again nonnegative by convexity, but since the ratio $s''(\hat{\Lambda}_{j}^n) / s''(\xoverline{\Lambda}_{j}^n)$ can certainly be larger than one, the standard CFL condition \eqref{eq:CFL} cannot be applied to show nonnegativity of the second factor in \eqref{eq:noMono}.

There are now two ways to achieve monotonicity despite the optimization error.

\subsection{Modifying the CFL condition}
\label{sec:modifiedCFL}

The more precisely the numerical optimizer solves the optimization problem, the smaller the ratio $s''(\hat{\Lambda}_{j}^n) / s''(\xoverline{\Lambda}_{j}^n)$ becomes.
This suggests using it as a stopping criterion and then incorporating it into a modified CFL condition. Summing up the findings from subsection~\ref{sec:Monotonicity}, we obtain the following theorem:

\begin{theorem}\label{th:CFLOrderOne}
Assume that the entropy ansatz only takes values in $(u_-, u_+)$ and that $g$ gives a monotone scheme.
Then when the numerical optimizer enforces the stopping criterion
\begin{align}\label{eq:modCFL}
 \frac{s''(\hat{\Lambda}_{j}^n)}{s''(\xoverline{\Lambda}_{j}^n)}\le \gamma,
\end{align}
the new moment vector $\bm{u}^{n + 1}_j$ computed by \eqref{eq:schemeInexact1} (i.e., \eqref{eq:implementedUpdate}) is realizable under the modified CFL condition
\begin{align}\label{eq:CFLfirstOrder}
\gamma\frac{\Delta t}{\Delta x} \max_{u\in[u_-,u_+]} \vert f'(u)\vert \le 1.
\end{align}
\end{theorem}

The condition \eqref{eq:modCFL} can be used instead of or in addition to \eqref{eq:tauCrit}.
The user chooses the parameter $\gamma$.
Larger values of $\gamma$ make the condition easier to fulfill, i.e., require fewer optimization iterations, but come at the cost of requiring smaller time steps and leading to more diffusive solutions.

Note that the condition \eqref{eq:modCFL} uses the unknown exact dual state $\hat{\Lambda}_{j}^n$. An approximation of this state can be constructed by making use of the Newton step. If $\bm{\lambda}$ is an iterate of the optimization method, $\bm{H}$ is the Hessian and $\bm{g}$ is the gradient of the dual problem \eqref{eq:dualProblem}, we approximate $\bm{\hat\lambda}$ by
\begin{align}\label{eq:lambda-hat-est}
\bm{\hat\lambda} \approx \bm{\lambda} - \zeta \bm{H}^{-1}(\bm{\lambda}) \bm{g}(\bm{\lambda}).
\end{align}
A safety parameter $\zeta$ is used to prevent underestimating the ratio ${s''(\hat{\Lambda}_{j}^n) / s''(\xoverline{\Lambda}_j^n)}$.

There are potential drawbacks of using Algorithm \ref{alg:seq} with the modified CFL condition \eqref{eq:CFLfirstOrder}.
First, it further restricts the time step, which introduces numerical diffusion.
Second, it's not immediately clear if we can practically achieve a stopping criterion of the form \eqref{eq:modCFL} for a reasonably small value of $\gamma$. 
Third, the choice $\Delta u=0$ is prohibited if the initial condition takes on values of $\min u_0$ or $\max u_0$ on a nonzero measure.
This is because in this case, the correct dual state $\hat{\Lambda}$ goes to infinity, leading to an infinite value of $\gamma$ no matter how precisely the numerical optimizer solves the dual problem.
We explore these potential problems in our numerical results in section \ref{sec:results}.

\begin{remark}
The issue of realizability is also an issue for minimum-entropy methods in kinetic theory.
A realizability-preserving modified CFL condition similar to the one presented in Theorem \ref{th:CFLOrderOne} was derived in \cite{alldredge2012high}.
But in kinetic theory, one only needs to ensure the nonnegativity of the underlying update, whereas we need to enforce both upper and lower bounds.
Because of this difference the modified CFL condition from \cite{alldredge2012high} is not enough to ensure realizability for the IPM method.
\end{remark}
\subsection{Modifying the scheme}
\label{sec:replace-moments}

Another way to prevent the optimization error from destroying the monotonicity properties of the underlying scheme is to remove the optimization error completely from our application of the underlying scheme, so that the ratio ${s''(\hat{\Lambda}_j^n) / s''(\xoverline{\Lambda}_j^n)}$ doesn't even appear in \eqref{eq:dHdLambdaj}.
Since the exact dual state cannot be computed, this means using the dual states $\xoverline{\Lambda}_j^n$ also in the first term of $H_j^n$.

More specifically, let us define the modified update function
\begin{align}
 \widetilde H(\xoverline \Lambda_{j - 1}^n, \xoverline \Lambda_j^n,
   \xoverline \Lambda_{j + 1}^n)
  &:= u_{ME}(\xoverline \Lambda_j^n) \\
  &\qquad - \frac{\Delta t}{\Delta x} \left( 
   g(u_{ME}(\xoverline \Lambda_j^n),
   u_{ME}(\xoverline \Lambda_{j+1}^n))
   - g(u_{ME}(\xoverline \Lambda_{j-1}^n),
   u_{ME}(\xoverline \Lambda_{j}^n)) \right). \nonumber
\end{align}
(Notice that since the optimization error plays no role, the modified update function itself no longer depends on the spatial cell or time.)
Now $\widetilde H$ immediately inherits the monotonicity properties of $h$ in \eqref{eq:h} under the original CFL condition \eqref{eq:CFL}, no matter how big or small the optimization error is.
The algorithm can be written in the original form as
\begin{align}
 \bm{u}_{j}^{n+1} = \xoverline{\bm{u}}_{j}^{n}
  - \frac{\Delta t}{\Delta x} \left(
  \langle g(u_{ME}(\xoverline{\Lambda}_j^n),
  u_{ME}(\xoverline{\Lambda}_{j+1}^n)) \bm{\varphi} \rangle
  - \langle g(u_{ME}(\xoverline{\Lambda}_{j-1}^n),
  u_{ME}(\xoverline{\Lambda}_{j}^n)) \bm{\varphi} \rangle\right),
\end{align}
and we present it in Algorithm \ref{alg:seqMod}.

\begin{algorithm}[H]
\begin{algorithmic}[1]
\FOR{$j=0$ to $NCells+1$}
\STATE $\bm{u}_j^0 = \frac{1}{\Delta x} \int_{x_{j-1/ 2}}^{x_{j+1/ 2}} \langle u_0(x, \cdot) \bm{\varphi} \rangle dx$
\ENDFOR
\FOR{$n=0$ to $NTimeSteps$}
\FOR{$j=0$ to $NCells+1$}
\STATE $\xoverline{\bm{\lambda}}_j^n \approx \argmin_{\bm{\lambda}}  \left( \langle s_*(\bm{\lambda}^T \bm{\varphi})\rangle - \bm{\lambda}^T \bm{u}_j^n \right)$
\hfill such that \eqref{eq:tauCrit} holds
\STATE $\xoverline \Lambda_j^n = \left(\xoverline{\bm{\lambda}}_j^n\right)^T\bm{\varphi}$
\STATE $\xoverline{\bm{u}}_{j}^{n} = \langle u_{ME}(\xoverline \Lambda_j^n) \bm{\varphi} \rangle$
\ENDFOR
\FOR{$j=1$ to $NCells$}
\STATE $\bm{u}_{j}^{n+1} = \xoverline{\bm{u}}_{j}^{n}
  - \frac{\Delta t}{\Delta x} \left(
  \langle g(u_{ME}(\xoverline{\Lambda}_j^n),
  u_{ME}(\xoverline{\Lambda}_{j+1}^n)) \bm{\varphi} \rangle
  - \langle g(u_{ME}(\xoverline{\Lambda}_{j-1}^n),
  u_{ME}(\xoverline{\Lambda}_{j}^n)) \bm{\varphi} \rangle\right)$ 
\ENDFOR
\ENDFOR
\end{algorithmic}
\caption{Modified IPM algorithm}
\label{alg:seqMod}
\end{algorithm}

But of course, one cannot simply use any value of the optimization tolerance $\tau$ and expect to end up with accurate results.
However, when the numerical flux $\bm{G}$ is Lipschitz continuous in each argument with constant $K$, the error between the update of Algorithm \ref{alg:seqMod} and the exact update of \eqref{eq:exactUpdate} is simply $\mathcal{O}(\tau)$.
Indeed, let $c := \Delta t / \Delta x$; then we have
\begin{linenomath}\begin{align*}
\Bigg\| \bm{u}_j^n &-\frac{\Delta t}{\Delta x}\left(\bm{G}(\Lambda(\bm{u}_{j}^n),\Lambda(\bm{u}_{j+1}^n))-\bm{G}(\Lambda(\bm{u}_{j-1}^n),\Lambda(\bm{u}_{j}^n))\right) \\
&- \left( \xoverline{\bm{u}}_j^n -\frac{\Delta t}{\Delta x}\left(\bm{G}(\Lambda(\xoverline{\bm{u}}_{j}^n),\Lambda(\xoverline{\bm{u}}_{j+1}^n))-\bm{G}(\Lambda(\xoverline{\bm{u}}_{j-1}^n),\Lambda(\xoverline{\bm{u}}_{j}^n))\right) \right)\Bigg\| \le \left( 1 + 4cK \right) \tau.
\end{align*}\end{linenomath}
Therefore we simply need to choose $\tau$ with the order of accuracy of the one-step error, which in this case is $\mathcal{O}(\Delta t \Delta x) = \mathcal{O}(\Delta x^2)$.
Then the results computed by Algorithm \ref{alg:seqMod} have the same order of accuracy as those computed by the exact method.

The main drawback to Algorithm \ref{alg:seqMod} is that it is no longer in conservative form. However, when we take $\tau = \mathcal{O}(\Delta x^2)$, the nonconservative part vanishes as the grid is refined.
Furthermore, in our numerical results below we did not observe any large increases in error compared to the solutions computed using the method presented in section \ref{sec:modifiedCFL} with small values of $\gamma$.


\section{Extending the scheme to higher order}\label{sec:highOrder}
The main computational expense of minimum-entropy methods comes from the repeated numerical solution of the dual problem, which needs to be solved for every spatial cell.
With a high-order method, fewer spatial cells can achieve a desired level of accuracy.
In this section we show how to construct a realizability-preserving second-order method.

\subsection{Second-order spatial reconstruction}

First we give a stable second-order method for the original PDE \eqref{eq:origProblem} and then plug the entropy ansatz $u_{ME}$ into this method and integrate the equations against the basis functions $\bm{\varphi}$ to get a second-order method for the IPM system.
We start by defining a linear spatial reconstruction of the solution in each cell $j$ by $p_j^n(x) = u_j^n+(x-x_j)\sigma_j^n$.
Here $\sigma_j^n:=\sigma(u_{j-1}^n,u_j^n,u_{j+1}^n)$ is the slope of the reconstruction in cell $j$ at time step $t_n$.
We use the second-order stable \textit{minmod} slope, which is given by
\begin{linenomath}\begin{align*}
\sigma(u,v,w) = \frac{1}{\Delta x}\text{minmod}(w-v,v-u)
\end{align*}\end{linenomath}
with the minmod function
\begin{linenomath}\begin{align*}
\text{minmod}(a,b) = 
\begin{cases}
a  & \text{if } |a| < |b|, ab>0 \\
b  & \text{if } |b| < |a|, ab > 0 \\
0  & \text{else}
\end{cases}.
\end{align*}\end{linenomath}
The reconstructions give cell edge values
\begin{align}
 u_{j \pm 1/2}^{n, \mp} := u_j^n \pm \sigma_j^n\frac{\Delta x}2,
\end{align}
which are inserted into the numerical flux to give the time update:
\begin{align}\label{eq:schemeGHighOrder}
u_j^{n + 1} = u_j^n-\frac{\Delta t}{\Delta x}(g(u_{j+1/2}^-,u_{j+1/2}^+)-g(u_{j-1/2}^-,u_{j-1/2}^+)).
\end{align}

When we use the slopes given by the minmod limiter, the reconstructions further have the property that the edge values are bounded by the values of the cell means.
This property is crucial for realizability.%
\footnote{
For other slopes which do not have this property, one would have to implement a bound-preserving limiter, see e.g. \cite{liu1996nonoscillatory}.
}

In applying this scheme to the IPM method, we compute the slopes point-wise in $\xi$ using the entropy ans\"atze computed by the numerical optimizer, i.e.,
\begin{align}
 \sigma_j^n = \sigma(u_{ME}(\xoverline \Lambda_{j - 1}^n),
  u_{ME}(\xoverline \Lambda_j^n), u_{ME}(\xoverline \Lambda_{j + 1}^n)).
\end{align}
This gives the edge values
\begin{linenomath}\begin{align*}
u_{j \pm 1/2}^{n,\mp}(\xoverline{\Lambda}_{j-1}^n,\xoverline{\Lambda}_j^n,\xoverline{\Lambda}_{j+1}^n) &:= u_{ME}(\xoverline{\Lambda}_j^n) \pm \frac{\Delta x}{2}\sigma_j^n.
\end{align*}\end{linenomath}

Now we want to consider the monotonicity of the time update \eqref{eq:schemeGHighOrder} with respect to the dual states of the cell average and both sides of the neighboring edges.
For this we need to define the dual states of the edges,
\begin{align}\label{eq:edgeDualState}
\xoverline{\Lambda}_{j \pm 1/2}^{n,\mp} &:= s'(u_{j \pm 1/2}^{n,\mp}(\xoverline{\Lambda}_{j-1}^n,\xoverline{\Lambda}_j^n,\xoverline{\Lambda}_{j+1}^n)),
\end{align}
so that we can write \eqref{eq:schemeGHighOrder} applied to IPM as
\begin{subequations}\label{eq:schemeSecondOrderH}
\begin{align}
\bm{u}_j^{n+1} = \langle H_j^n(\xoverline{\Lambda}_{j}^n,\xoverline{\Lambda}_{j+1/2}^{n,-},\xoverline{\Lambda}_{j+1/2}^{n,+},\xoverline{\Lambda}_{j-1/2}^{n,-},\xoverline{\Lambda}_{j-1/2}^{n,+}) \bm{\varphi}\rangle,
\end{align}
where 
\begin{align}\label{eq:secondOrderH}
H_j^n(\xoverline{\Lambda}_{j}^n,&\xoverline{\Lambda}_{j+1/2}^{n,-},
  \xoverline{\Lambda}_{j+1/2}^{n,+},\xoverline{\Lambda}_{j-1/2}^{n,-},
  \xoverline{\Lambda}_{j-1/2}^{n,+})
 :=u_{ME}(\xoverline{\Lambda}_j^n+\Delta\Lambda_j^n)\\
  &- \frac{\Delta t}{\Delta x} \Big( g(u_{ME}(\xoverline{\Lambda}_{j+1/2}^{n,-}),u_{ME}(\xoverline{\Lambda}_{j+1/2}^{n,+}))  - g(u_{ME}(\xoverline{\Lambda}_{j-1/2}^{n,-}), u_{ME}(\xoverline{\Lambda}_{j-1/2}^{n,+})) \Big). \nonumber
\end{align}
\end{subequations}
After having derived this underlying scheme we can find a time-step restriction which ensures realizability.

\begin{theorem}
Assume that the entropy ansatz only takes values in $(u_-, u_+)$ and that $g$ gives a monotone scheme.
Then the time-updated moment vector $\bm{u}_j^{n + 1}$ from the second-order in space scheme \eqref{eq:schemeSecondOrderH} is realizable under the time-step restriction
\begin{align}\label{eq:CFLHigherOrder}
\widetilde{\gamma}\max_{u\in[u_-,u_+]} \vert f'(u)\vert \frac{\Delta t}{\Delta x}\leq \frac{1}{2}
\end{align}
where $\widetilde{\gamma}$ satisfies
\begin{align}\label{eq:gammaHigherOrder}
\max \left\{\frac{s''\left(\xoverline{\Lambda}_{j + 1/2}^{n,-}
  + \Delta \Lambda_{j + 1/2}^{n,-}\right)}
  {s''\left(\xoverline{\Lambda}_{j-1/2}^{n,+}\right)}, 
  \frac{s''\left(\xoverline{\Lambda}_{j - 1/2}^{n,+}
  + \Delta \Lambda_{j - 1/2}^{n,+}\right)}
  {s''\left(\xoverline{\Lambda}_{j+1/2}^{n,-}\right)} \right\} 
 \le \widetilde{\gamma}.
\end{align}
\end{theorem}

\begin{proof}
As in Proposition \ref{th:MinMax}, we show that $H_j^n(\xoverline{\Lambda}_{j}^n,\xoverline{\Lambda}_{j+1/2}^{n,-},\xoverline{\Lambda}_{j+1/2}^{n,+},\xoverline{\Lambda}_{j-1/2}^{n,-},\xoverline{\Lambda}_{j-1/2}^{n,+})$ increases monotonically in every argument.

We show monotonicity by adopting the technique of writing $H_j^n$ as a convex combination of evaluations of the first-order scheme of \eqref{eq:schemeInexact1}
\cite{liu1996nonoscillatory}.
We write
\begin{align}
u_{ME}(\xoverline{\Lambda}_j^n+\Delta\Lambda_j^n) =
u_{ME}(\hat{\Lambda}_j^n) = \frac12\left(u_{ME}(\hat \Lambda_{j-1/2}^{n,+})+u_{ME}(\hat \Lambda_{j+1/2}^{n,-}) \right),
\end{align}
where $\hat \Lambda_{j \pm 1/2}^{n,\mp}$ are defined as in \eqref{eq:edgeDualState} but with $(\xoverline{\Lambda}_{j-1}^n,\xoverline{\Lambda}_j^n,\xoverline{\Lambda}_{j+1}^n)$ replaced by $(\hat{\Lambda}_{j-1}^n,\hat{\Lambda}_j^n,\hat{\Lambda}_{j+1}^n)$,%
\footnote{
In words, $\hat \Lambda_{j \pm 1/2}^{n,\mp}$ are derived from the pointwise linear reconstruction using the values of the exact entropy ansatz instead of the approximate entropy ansatz returned by the numerical optimizer.
}
and insert this into \eqref{eq:secondOrderH}, so that after adding and subtracting
\begin{linenomath}\begin{align*}
\frac{\Delta t}{\Delta x} g\left(u_{ME}\left(\xoverline{\Lambda}_{j-1/2}^{n,+}\right), u_{ME}\left(\xoverline{\Lambda}_{j+1/2}^{n,-}\right)\right)
\end{align*}\end{linenomath}
we can write $H_j^n$ as
\begin{linenomath}\begin{align*}
&H_j^n(\xoverline{\Lambda}_{j}^n,\xoverline{\Lambda}_{j+1/2}^{n,-},\xoverline{\Lambda}_{j+1/2}^{n,+},\xoverline{\Lambda}_{j-1/2}^{n,-},\xoverline{\Lambda}_{j-1/2}^{n,+})\\
&\quad=\frac{1}{2}\left( H_{j,+}^n(\xoverline{\Lambda}_{j-1/2}^{n,+},\xoverline{\Lambda}_{j+1/2}^{n,-},\xoverline{\Lambda}_{j+1/2}^{n,+})
 + H_{j,-}^n(\xoverline{\Lambda}_{j-1/2}^{n,-},\xoverline{\Lambda}_{j-1/2}^{n,+},
 \xoverline{\Lambda}_{j+1/2}^{n,-})\right),
\end{align*}\end{linenomath}
where
\begin{linenomath}\begin{align*}
H_{j,\pm}^n(\Lambda_\ell, \Lambda_c, \Lambda_r) := &u_{ME}(\Lambda_c + \Delta\Lambda_{j\pm 1/2}^{\mp})\\
& - 2\frac{\Delta t}{\Delta x}\left( g(u_{ME}(\Lambda_c),u_{ME}(\Lambda_r))-g(u_{ME}(\Lambda_\ell),u_{ME}(\Lambda_c)) \right)
\end{align*}\end{linenomath}
and
\begin{align}
\Delta \Lambda_{j \pm 1/2}^{n,\mp} := \hat \Lambda_{j \pm 1/2}^{n,\mp}
 - \xoverline{\Lambda}_{j \pm 1/2}^{n,\mp}.
\end{align}
The functions $H_{j,\pm}^n$ are similar to the first-order update function \eqref{eq:HfirstOrder}, so that one readily recognizes that they are monotone in each argument under the conditions
\begin{align}
 2 \frac{s''\left(\xoverline{\Lambda}_{j \pm 1/2}^{n,\mp}
  + \Delta \Lambda_{j \pm 1/2}^{n,\mp}\right)}
  {s''\left(\xoverline{\Lambda}_{j \pm 1/2}^{n,\mp}\right)}
  \max_{u\in[u_-,u_+]} |f'(u)|
  \frac{\Delta t}{\Delta x}
 \leq 1,
\end{align}
respectively.

Thus $H_{j,\pm}^n$ are both monotone under \eqref{eq:CFLHigherOrder}, so is $H_j^n$, and the realizability of $\bm{u}_j^{n + 1}$ follows. \qed
\end{proof}

Unfortunately the bound \eqref{eq:gammaHigherOrder} requires the coupling of the stopping criterion of the optimization problems in neighboring cells.
This destroys parallelizability.

We avoid this by adopting the same strategy as in Section \ref{sec:replace-moments}:
that is, we replace $u_{ME}(\xoverline{\Lambda}_j^n + \Delta \Lambda_j^n)$ in \eqref{eq:secondOrderH} with $u_{ME}(\xoverline{\Lambda}_j^n$).
Thus we do not have to consider the optimization error when checking monotonicity, and we get monotonicity under the condition
\begin{align}
 \max_{u\in[u_-,u_+]} \vert f'(u)\vert \frac{\Delta t}{\Delta x}\leq \frac{1}{2}.
\end{align}
In order to maintain accuracy, we use the stopping criterion \eqref{eq:tauCrit} with $\tau = \mathcal{O}(\Delta x^3)$.

\subsection{Second-order time integration}

For time integration we use strong stability-preserving (SSP) methods.  These are the standard choice for hyperbolic equations and allow us to build on our analysis of forward Euler steps, since SSP methods can be written as convex combinations of forward Euler steps.
We rewrite a forward Euler step in the form
\begin{align}
 \bm{u}_j^{n + 1} = \bm{u}_j^n + \Delta t L_j(\xoverline{\Lambda}_{j - 2}^n,
  \xoverline{\Lambda}_{j - 1}^n,\xoverline{\Lambda}_j^n,
  \xoverline{\Lambda}_{j + 1}^n, \xoverline{\Lambda}_{j + 2}^n),
\end{align}
where
\begin{linenomath}\begin{align*}
L_j(\xoverline{\Lambda}_{j - 2}^n, \xoverline{\Lambda}_{j - 1}^n,
   \xoverline{\Lambda}_j^n, \xoverline{\Lambda}_{j + 1}^n,
   \xoverline{\Lambda}_{j + 2}^n)
  := - \frac{1}{\Delta x} \Bigg( & \left\langle
   g(u_{ME}(\xoverline{\Lambda}_{j+1/2}^-),
   u_{ME}(\xoverline{\Lambda}_{j+1/2}^+)) \bm{\varphi} \right\rangle \\
  &- \left\langle g(u_{ME}(\xoverline{\Lambda}_{j-1/2}^-),
  u_{ME}(\xoverline{\Lambda}_{j-1/2}^+))\bm{\varphi}\right\rangle\Bigg).
\end{align*}\end{linenomath}
In particular, we use multistep SSP methods \cite{shu1988total}.  With multistep methods, we are able to re-use the evaluations of $L_j$ from previous time steps, in contrast to single-step (i.e., multistage Runge--Kutta) methods, which require multiple evaluations of $L_j$ for each time step.
The time update for a general multistep SSP method has the form
\begin{linenomath}\begin{align*}
\bm{u}^{n+1}_j = \sum_{i=1}^m \alpha_i \bm{u}^{n+1-i}_j + \Delta t \beta_i L_j(\xoverline{\Lambda}_{j - 2}^n, \xoverline{\Lambda}_{j - 1}^n,
   \xoverline{\Lambda}_j^n, \xoverline{\Lambda}_{j + 1}^n,
   \xoverline{\Lambda}_{j + 2}^n),
\end{align*}\end{linenomath}
where $m$ is the number of past steps used to compute the $(n + 1)$-th time step.
When a forward Euler step remains realizable under time step $\Delta t_{\rm FE}$, then the multistep SSP method remains realizable under time step $c\Delta t_{\rm FE}$, where
\begin{linenomath}\begin{align*}
 c:= \min_{\{i : \beta_i > 0\}}\frac{\alpha_i}{\vert \beta_i \vert}.
\end{align*}\end{linenomath}
We use the four-step second-order method found in \cite{gottlieb2001strong}:
\begin{align}\label{eq:sspMultistep2}
&\bm{\alpha} = \left(\frac{8}{9},0,0,\frac{1}{9}\right)^T,\enskip \bm{\beta} = \left( \frac{4}{3},0,0,0\right)^T,\enskip c = \frac{2}{3}. 
\end{align}
With this multistep SSP method, the new CFL condition is given by
\begin{align}\label{eq:CFLboundPresSSP}
\max_{u\in[u_-,u_+]} \vert f'(u)\vert \frac{\Delta t}{\Delta x}\leq \frac{1}{3}.
\end{align}

\section{Choosing the Entropy}\label{sec:fd-entropy}

While the log-barrier does the job of enforcing bounds on the oscillations around $\min u_0$ and $\max u_0$, it is not the only choice which achieves such bounds.
If we look at the form of the entropy ansatz $u_{ME}(\Lambda) = (s')^{-1}(\Lambda)$ in \eqref{eq:ansatz}, we see that it is sufficient that the derivative $s'$ maps the open interval $(u_-, u_+)$ to the entire real line.
I.e., it suffices that
\begin{align}\label{eq:slopeInf}
 \lim_{u \nearrow u_+} s'(u) \rightarrow \infty
 \quad \text{and} \quad
 \lim_{u \searrow u_-} s'(u) \rightarrow -\infty
\end{align}
to achieve the desired bounds on the entropy ansatz.
We can use this to find a new entropy with better properties.

In choosing an entropy, our goals are to satisfy the original maximum principle as closely as possible and to obtain a solution which oscillates as little as possible.
The first step is ensured by condition \eqref{eq:slopeInf}, as long as we take $\Delta u = u_+ - \max u_0 = \min u_0 - u_-$ as small as possible.
In fact, ideally we would like to just choose $\Delta u = 0$.

The log-barrier entropy \eqref{eq:log-barrier} achieves condition \eqref{eq:slopeInf} indirectly: by ruling out values outside of $(u_-, u_+)$ using barriers in $s$ itself.
There exist, however, moment vectors for which the ansatz must take on the value $\max u_0$ or $\min u_0$ on sets of nonzero measure.
The moment vectors of the initial condition $\bm{u}(0, x) = \langle u_0(x, \cdot) \bm{\varphi} \rangle$ take on such values when, for example, $u_0$ attains its maximum or minimum (over all $x \in \mathcal{D}$ and $\xi \in \Theta$) at some point in space with certainty (i.e., constant in $\xi$).
These moments lie on the boundary of the set of realizability when $\Delta u = 0$,
but since the realizable set is open, here the optimization problem has no solution.
In the limit as a sequence of moment vectors approaches such a nonrealizable moment, the corresponding limit of entropy ans\"atze does converge, but the entropy value $\langle s(\mathcal{U}_{ME}(\bm{u}) \rangle$ goes to infinity.
In this sense, the log-barrier entropy does not always recover the certain case gracefully.

But we can fulfill condition \eqref{eq:slopeInf} without forcing $s$ itself to take infinite values.
An entropy which achieves this is
\begin{align}\label{eq:BBEntropy}
s(u) = (u-u_-)\ln(u-u_-)+(u_+-u)\ln(u_+-u).
\end{align}
With $u_- = 0$ and $u_+ = 1$, this is the entropy for particles with Fermi--Dirac statistics. In the following, this entropy is called bounded-barrier (BB) entropy. It satisfies condition \eqref{eq:slopeInf} but is finite on the interval $[u_-, u_+]$. We compare the two entropy functions in Figure \ref{fig:comparisonEntropies}.

When one also compares the entropy ans\"atze resulting from the log-barrier and BB entropies in Figure \ref{fig:comparisonApproximation}, an interesting difference sticks out:
Here the BB entropy not only gives a much better solution, but in contrast to the solution using the log-barrier entropy it is not oscillatory around the value $u = (u_+ + u_-)/2 =: u_M$.
Further consideration of the shapes of the entropy functions offers a possible explanation.
In Figure \ref{fig:comparisonEntropies} we notice that the log-barrier entropy is much flatter than the BB entropy around their minimum at $u = u_M$.
Thus the log-barrier entropy does not distinguish among these values very well, and as a result the oscillations in its entropy ansatz seen in Figure \ref{fig:comparisonApproximation} are allowed because they have only a small effect on the value of the entropy.
Correspondingly, values near the boundaries of the domain $u_-$ and $u_+$ are strongly punished by the log-barrier entropy; this is in contrast to the bounded-barrier entropy, which simply takes finite values even at the end points.

\begin{figure}
\centering
\begin{subfigure}{.5\textwidth}
  \centering
  \includegraphics[width=1.0\linewidth]{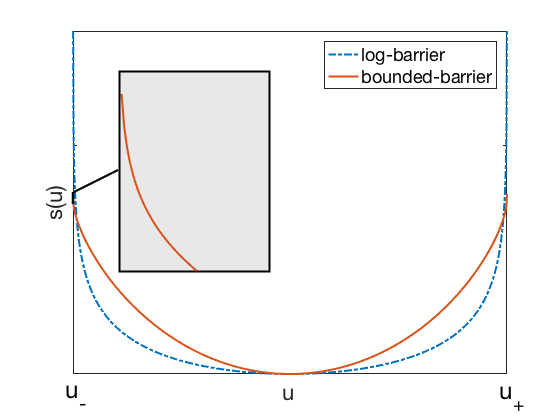}  
  \caption{}
  \label{fig:comparisonEntropies}
\end{subfigure}%
\begin{subfigure}{.5\textwidth}
  \centering
  \includegraphics[width=1.0\linewidth]{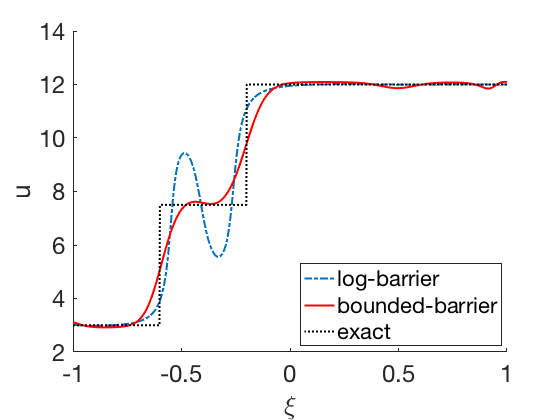}
    \caption{}
     \label{fig:comparisonApproximation}
\end{subfigure}
\caption{ (a) Comparison of entropies and (b) resulting approximation with $\Delta u = 0.1, N=10$.}
\label{fig:XiSpace}
\end{figure}
\begin{figure}[h!]
\centering
\begin{subfigure}{.5\textwidth}
  \centering
  \includegraphics[width=1.0\linewidth]{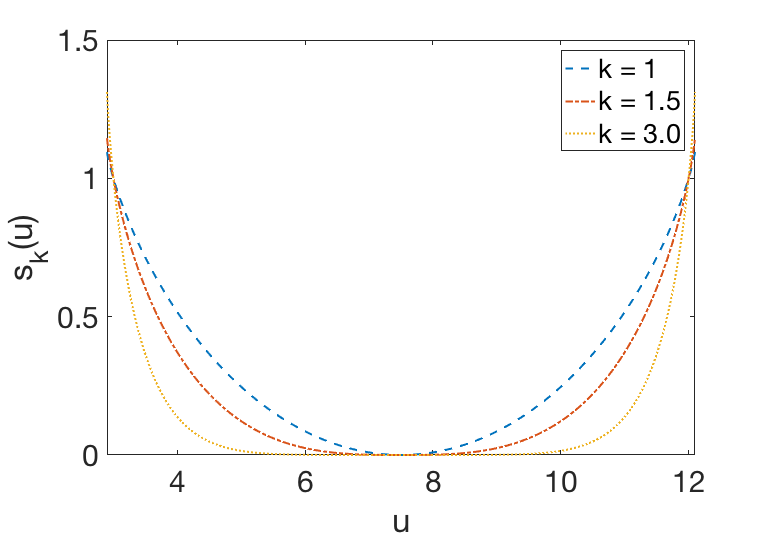}
  \caption{}
  \label{fig:Entropies}
\end{subfigure}%
\begin{subfigure}{.5\textwidth}
  \centering
  \includegraphics[width=1.0\linewidth]{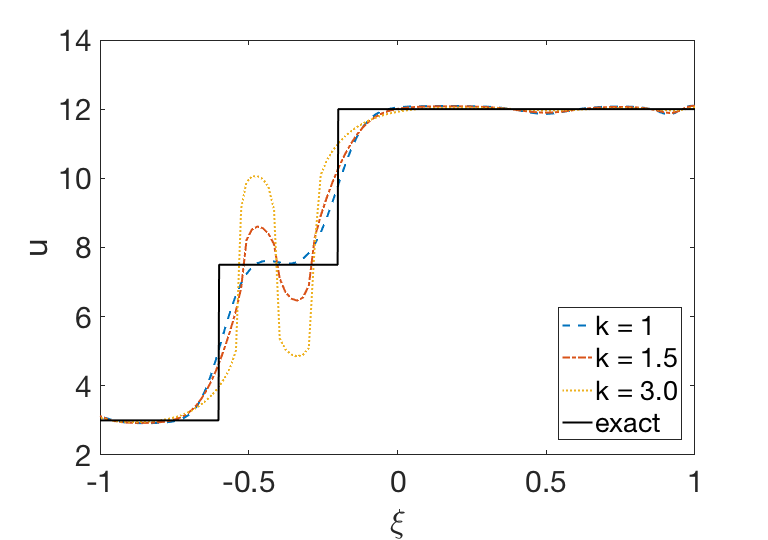}
  \caption{}
  \label{fig:Reconstructions}
\end{subfigure}
\caption{ (a) Family of entropies and (b) corresponding reconstruction for $\Delta u = 0.1, N=10$.}
\end{figure}
\begin{figure}[h!]
\centering
  \includegraphics[width=0.8\linewidth]{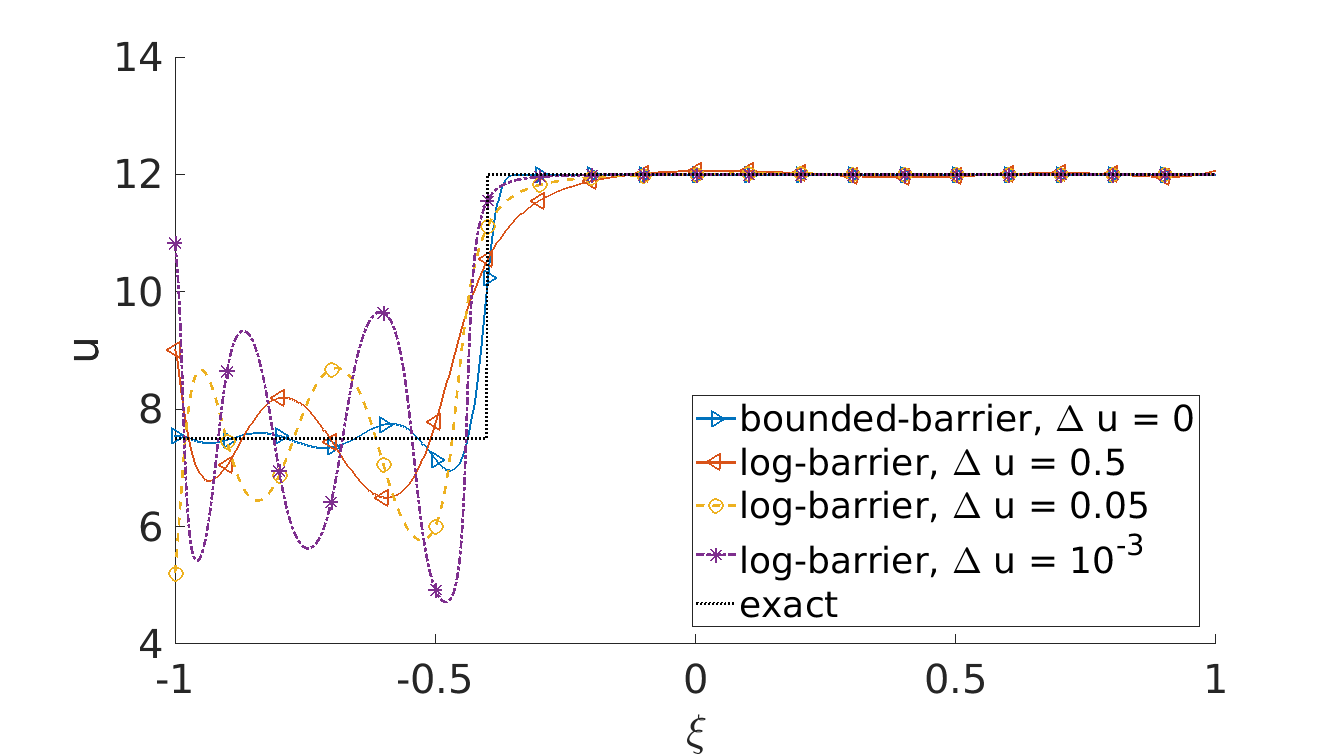}
  \caption{Approximation behavior for different values of $\Delta u$ with $N=10$.}
  \label{fig:logBarrierDeltaU}
\end{figure}

We tested this hypothesis by modifying the values of the slope around $u_M$ using the family of entropies
\begin{linenomath}\begin{align*}
s_{k}(u) = \left(\frac{s(u) - s\left(\frac{1}{2}(u_-+u_+)\right)}{s(u_{\max}) - s\left(\frac{1}{2}(u_-+u_+)\right)}\right)^k,
\end{align*}\end{linenomath}
where $s$ is the bounded-barrier entropy.
As we show in Figure \ref{fig:Entropies}, the higher $k$ is, the flatter the entropy is around $u_M$, so for higher values of $k$, we expect the entropy ansatz to be more oscillatory.
This is then exactly what we observe in Figure \ref{fig:Reconstructions}.

Another difference between the log- and bounded-barrier entropies is the dependence of the oscillations on the choice of $\Delta u$.
In numerical experiments, we noticed that with the log-barrier entropy, smaller values of $\Delta u$ are disadvantageous because the solutions are more oscillatory for smaller values of $\Delta u$.
We show an example of this behavior in Figure \ref{fig:logBarrierDeltaU}.
Here, we reconstruct a shock from $u_M$ to $u_{max}$.
The bounded-barrier entropy with $\Delta u = 0$ again gives the best result.
As we will see in the numerical results in the next section, the bounded-barrier entropy's more gentle behavior near the bounding values $u_-$ and $u_+$ allows us to choose $\Delta u = 0$ in all our numerical tests, thus exactly enforcing the original maximum principle.

\section{Numerical Results}
\label{sec:results}
In the following, we first compare the log-barrier and the bounded-barrier entropy in different test cases before turning to investigating the effectiveness of the two strategies to impose realizability. The exact solutions of all problems can be determined with the help of characteristics, see for example \cite[Chapter~3]{leveque1998nonlinear}. Furthermore, we use the upwind numerical flux in all test cases.

\subsection{Comparing different entropies}
\label{sec:UBurgers}
We start by comparing results when making use of the log- and bounded-barrier entropies. Following \cite{poette2009uncertainty}, we solve the uncertain Burgers' equation 
\begin{subequations}\label{eq:Burgers}
\begin{align}
\partial_t &u(t,x,\xi)+\partial_x \frac{u(t,x,\xi)^2}{2} = 0,\\
&u(t=0,x,\xi) = u_0(x,\xi),
\end{align}
\end{subequations}
with the first-order method in Algorithm \ref{alg:seqMod}.
As in \cite{poette2009uncertainty}, we choose the random initial condition
\begin{align}\label{eq:IC1}
u_0(x,\xi) &:= 
\begin{cases} u_L, & \mbox{if } x< x_0+\sigma\xi \\ u_L+\frac{u_R-u_L}{x_0-x_1} (x_0+\sigma \xi-x), & \mbox{if } x\in[x_0+\sigma \xi,x_1+\sigma \xi]\\
u_R, & \text{else }
\end{cases}
\end{align}
which is a forming shock with a linear connection from $x_0$ to $x_1$. In our case, $\xi$ is uniformly distributed on the interval $[-1,1]$. Due to the fact that we recalculate moments to ensure realizability, we can use the original CFL condition \eqref{eq:CFL}. We use the following parameter values:\\
\begin{center}
    \begin{tabular}{ | l | p{7cm} |}
    \hline
    $[a,b]=[0,3]$ & range of spatial domain \\
    $N_x=160$ & number of spatial cells \\
    $t_{end}=0.15$ & end time \\
    $x_0 = 0.5, x_1=1.5, u_L = 12, u_R = 3, \sigma = 0.2$ & parameters of initial condition \eqref{eq:IC1}\\
    $N+1 = 5$ & number of moments \\
    $\tau = 10^{-7}$ & gradient tolerance \eqref{eq:tauCrit} \\
    $\Delta u \in \{0, 0.001,0.5\}$ & distance $u_0$ to IPM bounds \\
    \hline
    \end{tabular}
\end{center}
Additionally, we computed all integrals in $\xi$ using a forty-point Gauss-Legendre quadrature.

\begin{figure}[h!]
\centering
  \includegraphics[width=0.8\linewidth]{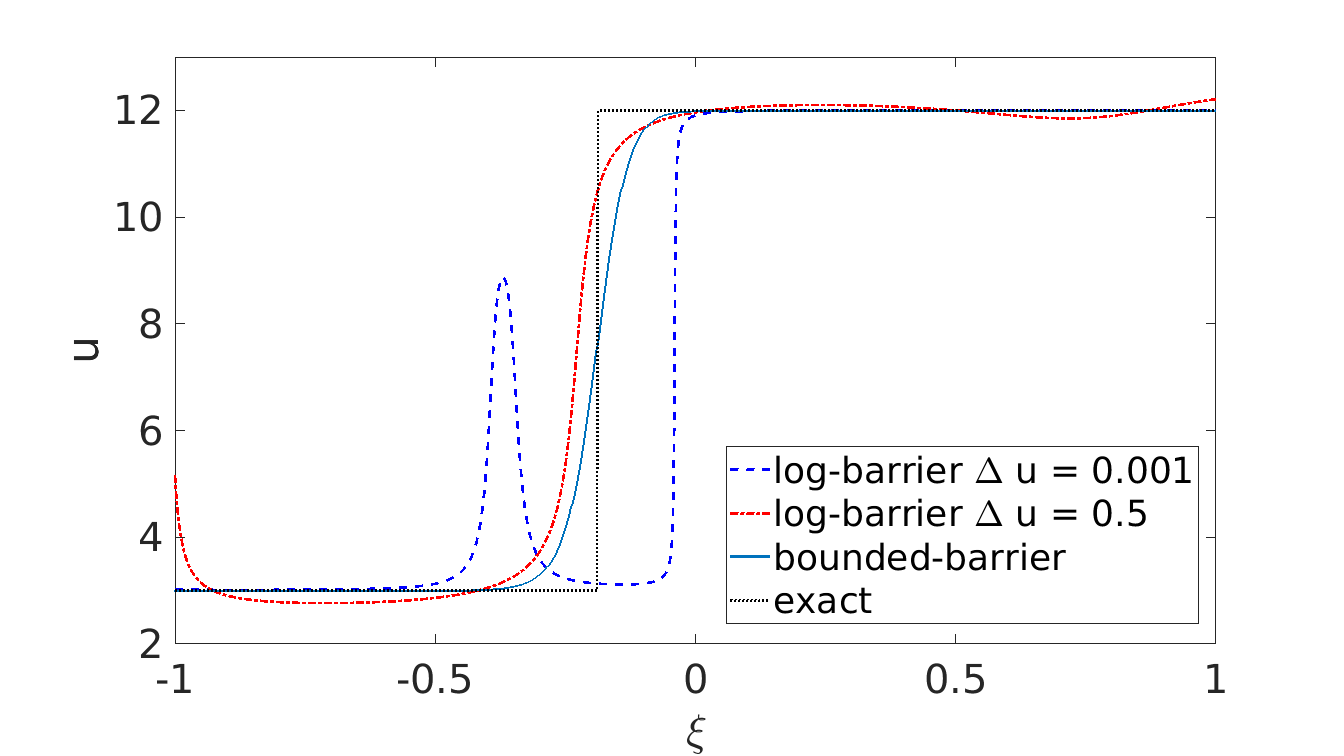}
  \caption{Solutions for log-barrier and bounded-barrier entropies at fixed spatial position $x$.}
  \label{fig:IC1fixedX}
\end{figure}
\begin{figure}[h!]
\centering
\begin{subfigure}{.5\textwidth}
  \centering
  \includegraphics[width=1.0\linewidth]{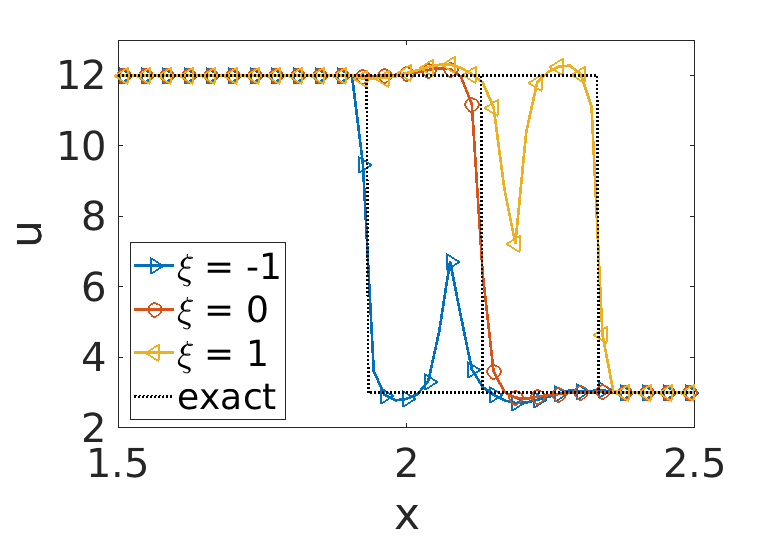}
  \caption{log-barrier entropy, $\Delta u = 0.5$.}
  \label{fig:u0Lambda}
\end{subfigure}%
\begin{subfigure}{.5\textwidth}
  \centering
  \includegraphics[width=1.0\linewidth]{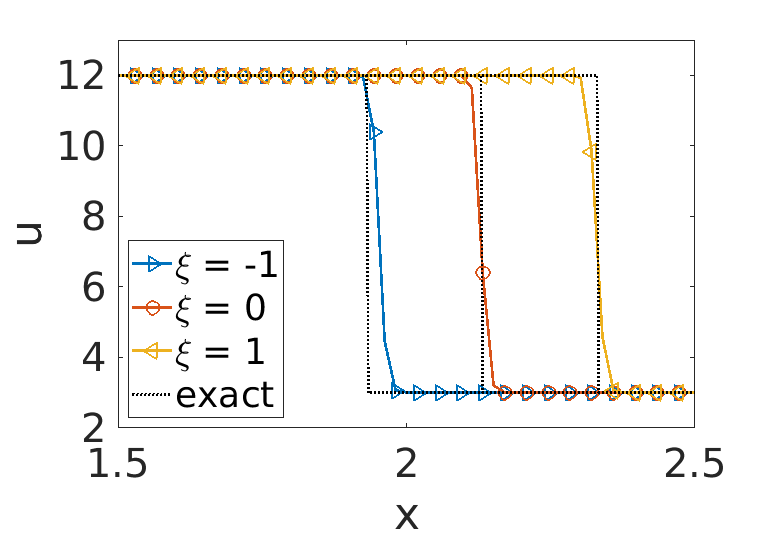}
  \caption{bounded-barrier entropy, $\Delta u = 0$.}
  \label{fig:XiSpaceT2}
\end{subfigure}
\caption{ Solution for different entropies evaluated at $\xi\in\{-1,0,1\}$. }
\label{fig:fixedXi}
\end{figure}
\begin{figure}[h!]
\centering
  \includegraphics[width=0.8\linewidth]{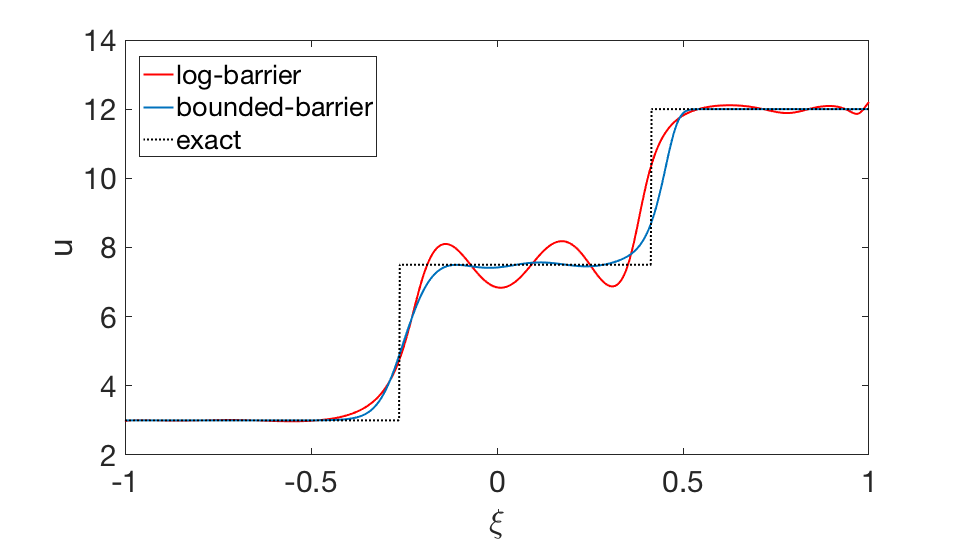}
  \caption{Solutions for log-barrier and bounded-barrier entropies at $x^* = 2.1$.}
  \label{fig:IC4fixedX}
\end{figure}
Since the log-barrier entropy is infinite at $u_+$ and $u_-$, we need to choose $\Delta u > 0$. We choose $\Delta u = 0.5$ as in \cite{poette2009uncertainty} as well as $\Delta u = 0.001$ to demonstrate the effects when the solutions lie close to the minimal and maximal value of the exact solution.
Note that the maximal velocity of the equation is $u_+ = u_L + \Delta u$, so consequently the CFL condition of the deterministic problem (where velocities are bounded by $u_L$) cannot be used.
The bounded-barrier entropy shows good approximation results for small values of $\Delta u$, so we set this parameter to zero, allowing the use of the deterministic CFL condition.
Plotting the solutions at fixed values for $\xi$ in Figure \ref{fig:fixedXi} shows the expected poor approximation behavior of the log-barrier entropy for small values of $\Delta u$. The choice $\Delta u = 0.5$ leads to over- and undershoots when using the log-barrier entropy, whereas the bounded-barrier entropy nicely approximates the solution.
Furthermore, the solution obtained with the bounded-barrier entropy fulfills the original maximum principle.
Looking at the dependency on $\xi$ for a fixed spatial cell in Figure \ref{fig:IC1fixedX}, one observes that the log-barrier entropy has oscillations whereas the bounded-barrier entropy gives a nonoscillatory solution.

Let us now turn to a new initial condition for the uncertain Burgers' equation in order to investigate the oscillations arising at a noncritical state $u_M$:
\begin{align}\label{eq:IC4}
u_0(x,\xi) &:= 
\begin{cases} u_L, & \mbox{if } x\leq x_0+\sigma\xi \\ u_L + (u_M-u_L)\cdot\frac{x_0+\sigma\xi-x}{x_0-x_1}, & \mbox{if } x\in( x_0+\sigma\xi, x_1+\sigma\xi]\\
u_M, & \mbox{if } x\in( x_1+\sigma\xi, x_2+\sigma\xi]\\
u_M + (u_R-u_M)\cdot\frac{x_3+\sigma\xi-x}{x_3-x_2}, & \mbox{if } x\in( x_2+\sigma\xi, x_3+\sigma\xi]\\
u_R, & \mbox{if } x > x_3+\sigma\xi,
\end{cases}
\end{align}
This initial condition describes two forming shocks that connect the three states $u_L$, $u_M$, and $u_R$.
All parameters which have been modified can be found in the following table:
\begin{center}
    \begin{tabular}{ | l | p{6.5cm} |}
    \hline
    $t_{end}=0.04$ & end time \\
    $x_0 = 0.8, x_1 = 0.98, x_2 = 1.32, x_3 = 1.5, \sigma = 0.5$ & parameters of initial condition \eqref{eq:IC4}\\
    $N+1 = 16$ & number of moments \\
    \hline
    \end{tabular}
\end{center}
The results for this problem can be seen in Figure \ref{fig:IC4fixedX}.
One observes that the solution using the log-barrier entropy is oscillatory, whereas with the bounded-barrier entropy the solution shows only small oscillations.
While the IPM scheme with the bounded-barrier entropy fulfills the maximum principle, the solution of the log-barrier entropy has over- and undershoots as large as $\Delta u$.

\subsection{Comparison of entropies in two-dimensional Random Space}
To compare both entropies in a two-dimensional random domain (i.e., $P = 2$), 
the initial condition of the Burgers' test case is changed to
\begin{align}\label{eq:IC3}
u_0(x) &:= 
\begin{cases} u_L+\sigma_0 \xi_0, & \mbox{if } x< x_0, \\ u_M+\sigma_1 \xi_1, & \mbox{if } x\in[x_0,x_1],\\
u_R, & \text{else,}
\end{cases}
\end{align}
where $\xi_0$ and $\xi_1$ are both uniformly distributed in $[-1,1]$.
This test case represents an uncertain multiple-shock flow, which is studied in 
compressible fluid mechanics, see \cite{poette2009uncertainty}. Realizability is 
again preserved by recalculating moments, meaning that the original CFL 
condition \eqref{eq:CFL} can be used. As in \cite{poette2009uncertainty}, a 
tensorized Clenshaw-Curtis quadrature rule of level 3 and an increased number of $N_x = 6000$ spatial grid points is used. In contrast to the other test cases, we need to choose a fine resolution of the spatial grid to minimize the effects of numerical diffusion, which significantly affects the solution in this test case.

\begin{figure}[h!]
\centering
	\begin{subfigure}{0.5\linewidth}
		\centering
		\includegraphics[scale=0.38]{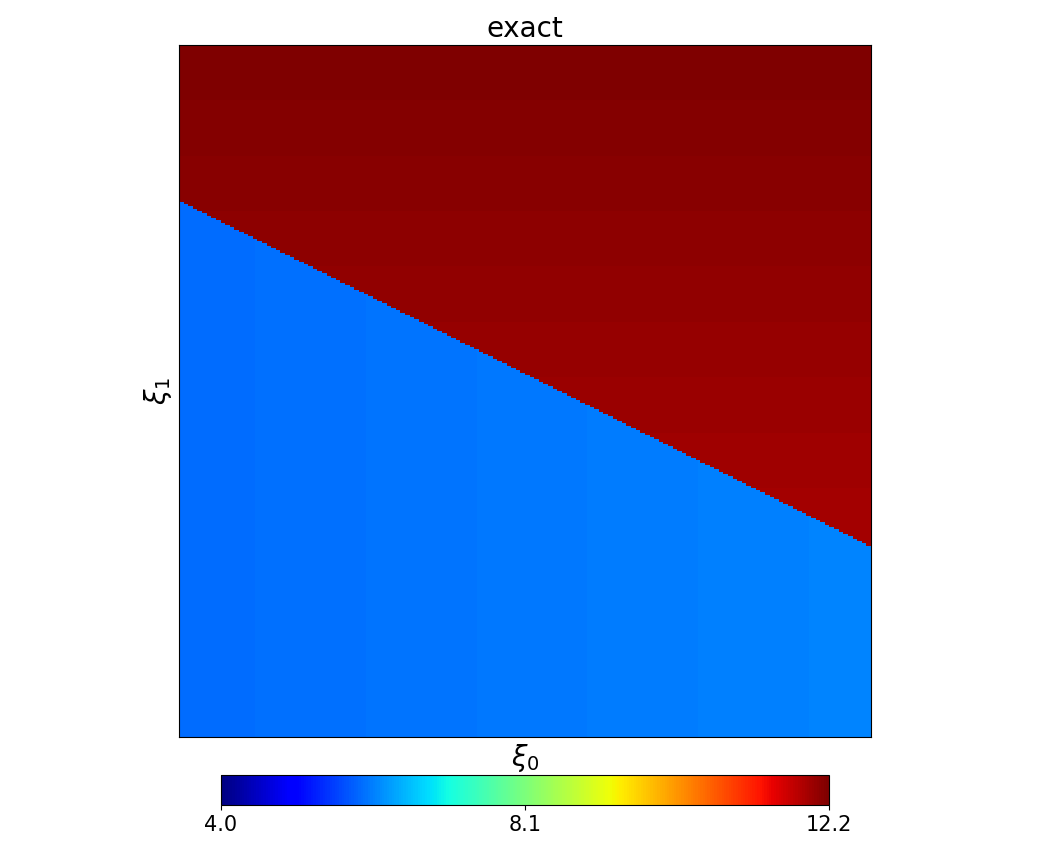}
		
		\label{fig:Burgers2DEx}
	\end{subfigure}%
	\begin{subfigure}{0.5\linewidth}
		\centering
		\includegraphics[scale=0.38]{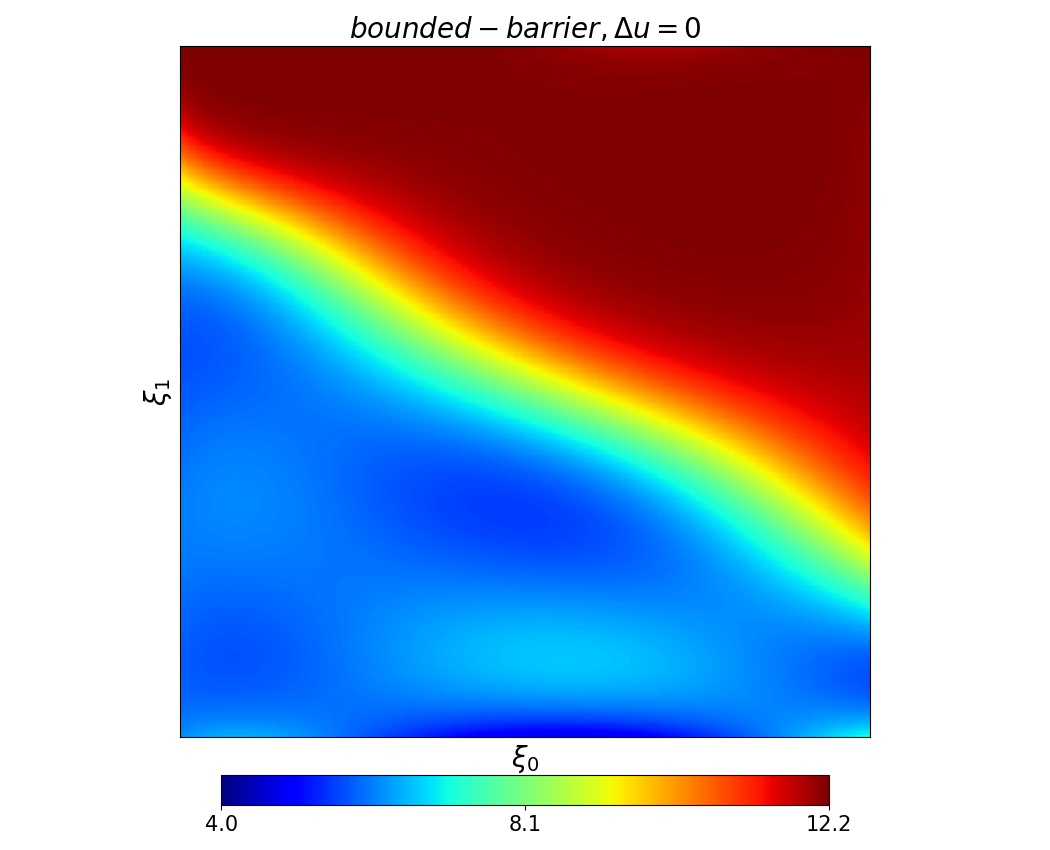}
		
		\label{fig:Burgers2DBB}
	\end{subfigure}
	\begin{subfigure}{0.5\linewidth}
		\centering
		\includegraphics[scale=0.38]{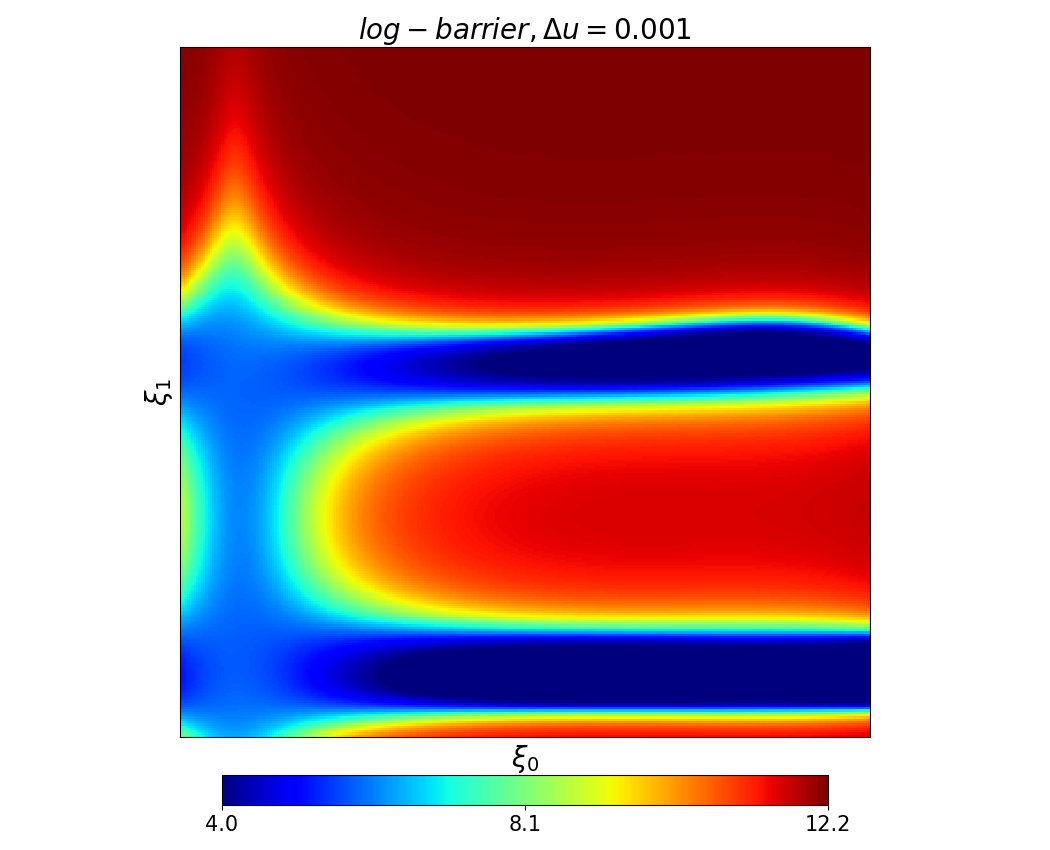}
		
		\label{fig:Burgers2DLBSmall}
	\end{subfigure}%
	\begin{subfigure}{0.5\linewidth}
		\centering
		\includegraphics[scale=0.38]{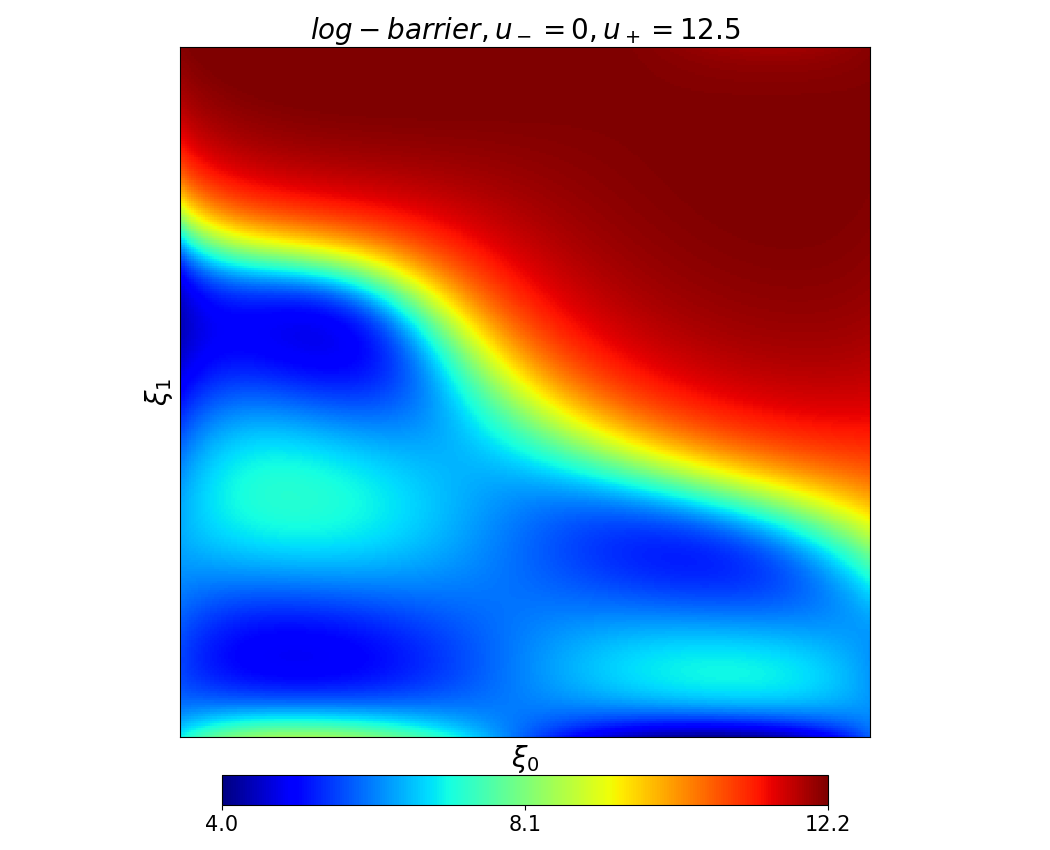}
		
		\label{fig:Burgers2DLB}
	\end{subfigure}
	\caption{Solution at $x^* = 0.4$ with different entropies}
	\label{fig:Burgers2D}
\end{figure}

\begin{center}
    \begin{tabular}{ | l | p{7cm} |}
    \hline
    $[a,b]=[0,1]$ & range of spatial domain \\
    $N_x=6000$ & number of spatial cells \\
    $t_{end}=0.01115$ & end time \\
    $x_0 = 0.3, x_1=1.6, \sigma_0 = 0.2,\sigma_1 = 0.2$, & parameters of initial condition \eqref{eq:IC3}\\
    $u_L = 12, u_M = 6, u_R = 1$ & \\
    $N+1 = 5$ & number of moments \\
    \hline
    \end{tabular}
\end{center}

The results are given in Figure~\ref{fig:Burgers2D}.
IPM again fulfills the maximum principle when the bounded-barrier entropy is 
used. The solution has only small oscillations around the intermediate state 
$u_M$ and shows good agreement with the exact solution. When trying to approach 
a maximum principle by choosing a small value of $\Delta u$ with the 
log-barrier entropy, the solution starts to oscillate heavily at the 
intermediate state. The solution resembles the one-dimensional result for a 
small value of $\Delta u$ depicted in Figure~\ref{fig:IC1fixedX}. Choosing the 
IPM bounds further away from the exact solution bounds (as in 
\cite{poette2009uncertainty}), we obtain a more accurate solution. However the 
maximum principle is not fulfilled since the solution takes on values bigger 
than $12.34$ (off the color scale) while showing oscillations at the 
intermediate state. This is also in agreement with the one-dimensional results 
shown before, where the maximum principle is violated by the log-barrier 
entropy.

\subsection{Convergence of different schemes}\label{sec:Convergence}
Due to its advantages compared to the log-barrier entropy, the following results have been obtained using the bounded-barrier entropy with $\Delta u = 0$. To investigate the convergence properties of the proposed first- and second-order schemes, we look at the advection equation with uncertain initial data
\begin{linenomath}\begin{align*}
\partial_t u(t,x,\xi)+\partial_x u(t,x,\xi) &= 0, \\
u(t=0,x,\xi) &= \sin(x+0.05\pi\xi),
\end{align*}\end{linenomath}
where $x\in[0,2]$ and $t_{end}=0.1$. We use periodic boundary conditions at the boundaries of the spatial domain. The number of moments we calculate is $3$.
We study the $L^1$ error of the expected value for different numbers of spatial discretization points.
Let $\bm{u}_h$ denote a numerical solution.
For first-order methods, it is constant across space in each spatial cell, and for second-order methods, it is defined according to the linear reconstructions given in Section \ref{sec:highOrder}.
Then we compute the $L^1$ error for each moment component by
\begin{linenomath}\begin{align*}
 \bm{e} := \int_0^2 |\bm{u}_h(t_{end}, x) - \bm{u}(t_{end}, x)| \,dx,
\end{align*}\end{linenomath}
where $\bm{u}(t_{end}, x)$ is the exact solution to the system of IPM moment equations \eqref{eq:sysMEU} at the final time $t_{end}$, and the absolute value and integral are taken component-wise.
In the following convergence results, we plot only the results for the zero-th component of $\bm{e}$.
The resulting convergence plot is given in Figure \ref{fig:convergence}.
\begin{figure}[h!]
\centering
\begin{subfigure}{.5\textwidth}
  \centering
  \includegraphics[width=1.0\linewidth]{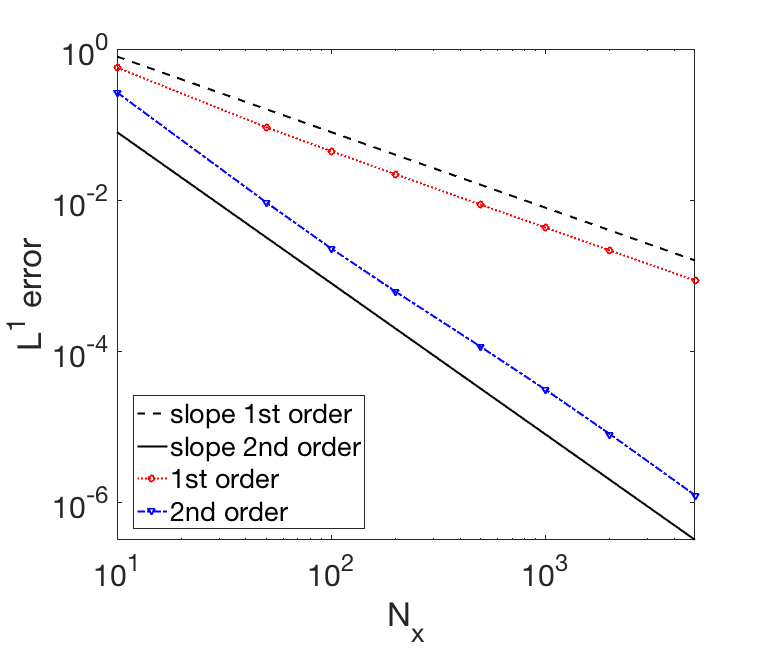}  
  \caption{}
  \label{fig:convergence}
\end{subfigure}%
\begin{subfigure}{.5\textwidth}
  \centering
  \includegraphics[width=1.0\linewidth]{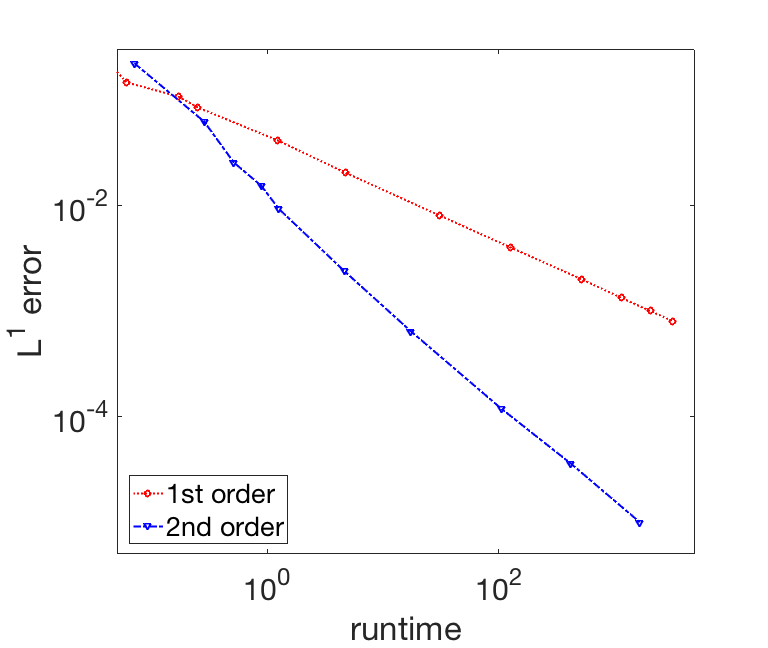}
    \caption{}
     \label{fig:Efficincy}
\end{subfigure}
\caption{ (a) Convergence of different IPM discretizations and (b) efficiency when using first- and second-order methods.}
\end{figure}
Both methods recalculate moments with the inaccurate dual states, meaning that in order to preserve the expected convergence rate $p\in\{1,2\}$, the stopping criterion of the optimization method needs to be set to $\tau = \Delta x^{p+1}$. For the time discretization of the second-order scheme, the four-step SSP scheme \eqref{eq:sspMultistep2} has been used.
Heun's method is used to calculate the first three time steps.
That the different schemes show the expected convergence.

The efficiency of the two methods shown in Figure \ref{fig:Efficincy} demonstrates that the second-order scheme reaches most levels of accuracy with less computing time than the first-order scheme.

\subsection{Comparison of strategies to preserve realizability}
\label{sec:CompareMethods}
Two strategies to ensure realizability have been presented in section \ref{sec:modifiedCFL} and section \ref{sec:replace-moments}, namely using a modified CFL condition or modifying moments. To compare these two strategies, we look at the uncertain advection equation given by
\begin{linenomath}\begin{align*}
\partial_t &u(t,x,\xi)+a(\xi) \partial_x u(t,x,\xi) = 0,\\
&u(t=0,x) = u_0(x).
\end{align*}\end{linenomath}
We choose $a(\xi):=11+\xi$, where $\xi$ is uniformly distributed on $[-1, 1]$, so the velocity is uniformly distributed in the interval $[10,12]$.
What is interesting about this equation is that the velocity of the system is not known, which means that our CFL condition adds artificial viscosity to smaller velocities, while high velocities are well resolved.
We use the deterministic initial condition
\begin{align}\label{eq:IC1Deterministic}
u_0(x) &:= 
\begin{cases} u_L, & \mbox{if } x< x_0, \\ u_L+\frac{u_R-u_L}{x_0-x_1} (x_0-x), & \mbox{if } x\in[x_0,x_1],\\
u_R, & \text{else.}
\end{cases}
\end{align}
Parameters of the calculation can be found in the following table:
\begin{center}
    \begin{tabular}{ | l | p{7cm} |}
    \hline
    $N_x=80$ & number of spatial cells \\
    $t_{end}=0.19$ & end time \\
    $x_0 = 0.5, x_1=0.55, u_L = 12, u_R = 3$ & parameters of initial condition \eqref{eq:IC1Deterministic}\\
    $N+1 = 10$ & number of moments \\
    $\gamma \in \{1.5,1.1,1+10^{-7}\}, \zeta = 5$ & CFL modification \\
    $\zeta = 5$ & safety factor in estimation of $\bm{\hat \lambda}$ \eqref{eq:lambda-hat-est} \\
    $\Delta u \in \{0,10^{-7}\}$ & distance $u_0$ to IPM bounds \\
    \hline
    \end{tabular}
\end{center}
When modifying moments, we perform the computation using a first-order scheme as well as with second-order spatial reconstructions using the minmod limiter.
Since the second-order time discretization adds artificial viscosity without improving the accuracy, we use the explicit Euler method.
We use $\Delta u = 10^{-7}$ so that we can achieve the stopping criterion \eqref{eq:CFLfirstOrder} on every quadrature point.
\begin{figure}[h!]
\centering
  \includegraphics[width=0.8\linewidth]{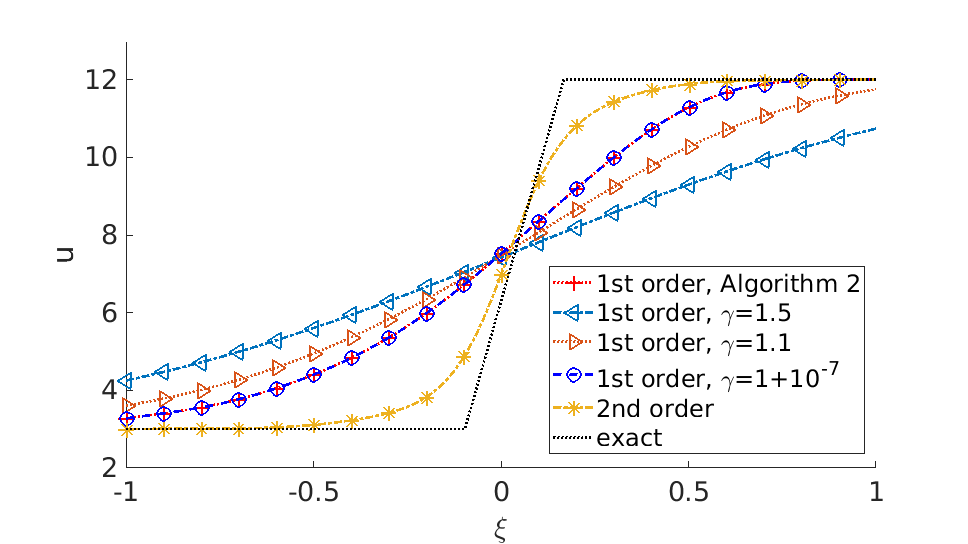}
  \caption{Solutions at $x^* = 2.6$ with and without using a spatial limiter.}
  \label{fig:AdvectionIC1SteepFixedX}
\end{figure}
\begin{figure}[h!]
\centering
  \includegraphics[width=0.8\linewidth]{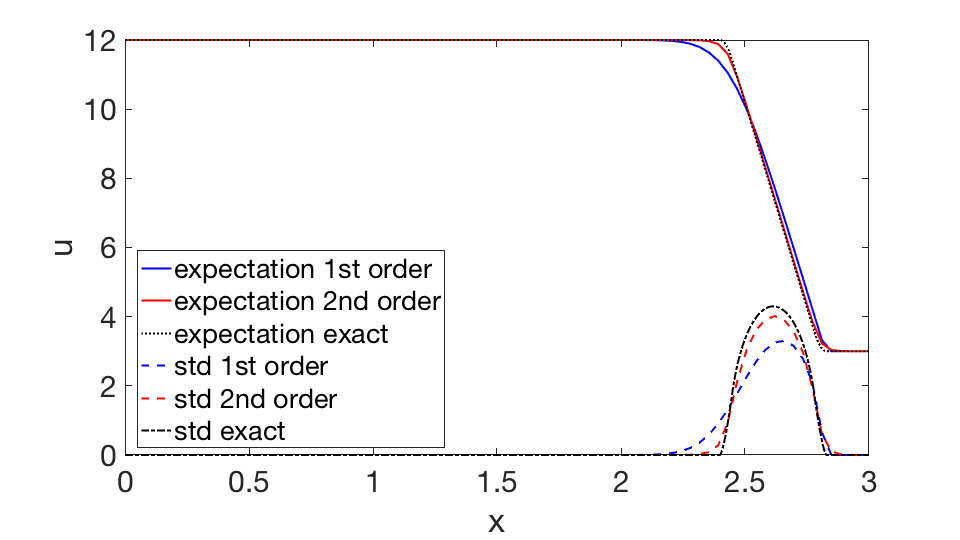}
  \caption{Expected value and standard deviation with and without limiters.}
  \label{fig:AdvectionIC1SteepExpectationVariance}
\end{figure}
Figure \ref{fig:AdvectionIC1SteepFixedX} shows the solution at a fixed position $x^*$.
Using Algorithm \ref{alg:seqMod} to ensure realizability allows the use of the deterministic CFL condition \eqref{eq:CFL}.
We compare this solution to those obtained with a modified CFL conditions according to section \ref{sec:modifiedCFL}.
As expected, modifying the CFL condition by $\gamma = 1.5$ leads to a heavily smeared-out solution.
However, we found that for this problem, it was possible to set $\gamma$ as small as $1 + 10^{-7}$.
The solution calculated with this value of $\gamma$ is essentially identical to the solution using Algorithm \ref{alg:seqMod}; they differ only on the order of $10^{-3}$ in the $L^\infty$ norm.
One can conclude that both realizability-preserving strategies for first-order methods are satisfactory.

Figure \ref{fig:AdvectionIC1SteepFixedX} also shows that the second-order spatial reconstructions give much better results.
This improvement is also seen in the expected value and standard deviation in Figure \ref{fig:AdvectionIC1SteepExpectationVariance}.
The standard deviation is particularly improved by going to second-order.

To underline the effects of artificial viscosity, we plot the solution when recalculating moments for first- and second-order spatial reconstructions for $\xi\in\{-1,0,1\}$.
Figure \ref{fig:AdvectionIC1SteepFixedXia} shows that the solution is well resolved if $\xi = 1$.
In the case of $\xi=-1$, the solution is smeared out, since the CFL condition does not allow the scheme to sharply capture shocks.
In Figure \ref{fig:AdvectionIC1SteepFixedXib}, we see that this effect is smaller when using second-order reconstructions.
\begin{figure}[h!]
\centering
\begin{subfigure}{.5\textwidth}
  \centering
  \includegraphics[width=1.0\linewidth]{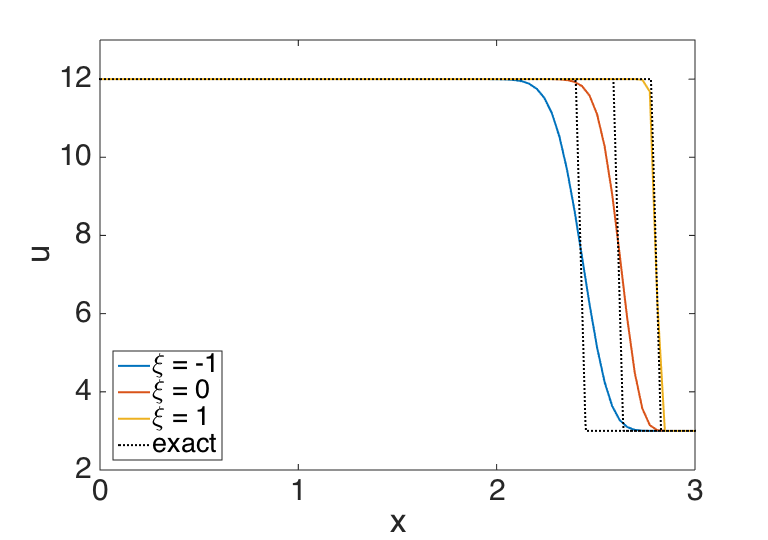}
  \caption{First order.}
  \label{fig:AdvectionIC1SteepFixedXia}
\end{subfigure}%
\begin{subfigure}{.5\textwidth}
  \centering
  \includegraphics[width=1.0\linewidth]{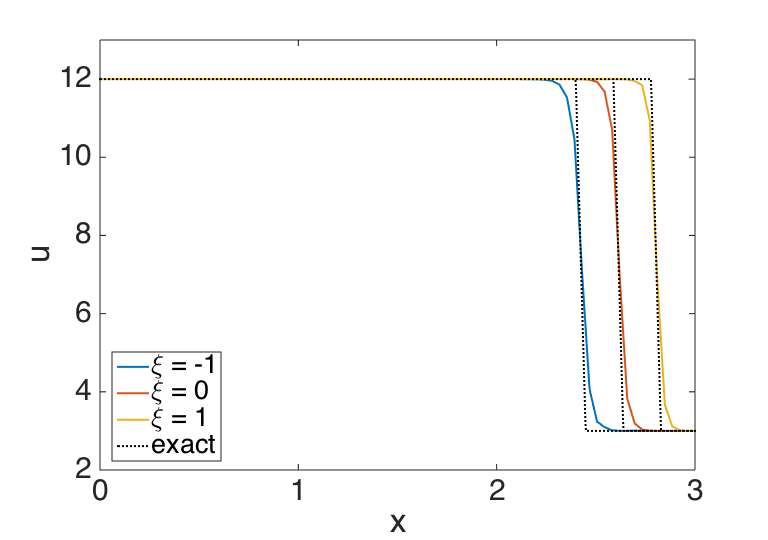}
  \caption{Second order.}
  \label{fig:AdvectionIC1SteepFixedXib}
\end{subfigure}
\caption{ Solution evaluated at $\xi\in\{-1,0,1\}$ with and without limiters.}
\label{fig:AdvectionIC1SteepFixedXi}
\end{figure}

\section{Conclusion and Outlook}\label{sec:Section8}
In this paper, we have investigated robust implementations of the IPM method for uncertain scalar hyperbolic conservation laws.
The standard discretization of the IPM moment system can easily lead to nonrealizable moments and these nonrealizable moments cause the numerical solver to crash because in these cases the ansatz is undefined.
This is especially true when the IPM bounds $u_-$ and $u_+$ are chosen very close to the bounds of the true solution.
In order to construct a second-order discretization of the IPM scheme that prevents such realizability problems, we investigated the numerical scheme in terms of the monotonicity of the underlying scheme for the original PDE.
We derived two first-order schemes which preserve realizability:
The first scheme makes use of a modified CFL condition and the second scheme recalculates moments from the inexact dual state.
We also extended this second scheme to second order.

We also investigated the approximation properties of the IPM scheme using different entropies.
By considering the entropy ansatz directly, we showed that the solution is not bounded due to properties of the entropy density $s$ itself but rather its derivative $s'$.
This allowed us to construct an entropy, which we called the bounded-barrier entropy, that takes finite values at the bounds $u_-$ and $u_+$
The bounded-barrier entropy behaves more gracefully near the boundary values $u_-$ and $u_+$, which we showed also leads to better solutions at intermediate values.
This allowed us to take the IPM bounds to be the minimal and maximal value of the true solution, thus allowing the method to fulfill the exact maximum principle of the underlying PDE.

We applied our numerical schemes to the uncertain Burgers' equation as well as the uncertain advection equation.
We observed that in contrast to solutions using the log-barrier entropy the solutions calculated using the bounded-barrier entropy fulfill the maximum principle and are nonoscillatory, particularly at intermediate states.

We consider the IPM method a promising tool to treat uncertain hyperbolic equations which is a clear improvement over the stochastic-Galerkin method.
In order to compete with the faster computation times of the stochastic-Galerkin method one should focus on accelerating the process of solving the dual problem, taking advantage of parallelizability, as well as higher-order schemes.
An extension to higher-order schemes should be straightforward with bound-preserving limiters \cite{liu1996nonoscillatory,zhang2010positivity}.

\section*{Acknowledgment} \noindent
This work was supported by the German Research Foundation (DFG).
Jonas Kusch and Martin Frank were supported under grant FR 2841/6-1 and
Graham Alldredge under AL 2030/1-1.

\bibliographystyle{siamplain}

\bibliography{IPMPaper}
\end{document}